\documentclass[final]{siamart251216}
\usepackage[english]{babel}
\usepackage[utf8]{inputenc}
\usepackage[T1]{fontenc}
\usepackage{lipsum}
\usepackage{amsfonts, amsmath, amssymb, mathtools}
\usepackage{graphicx}
\usepackage{cleveref}
\usepackage{enumerate}
\usepackage{epstopdf}
\usepackage{algorithmic}
\ifpdf
    \DeclareGraphicsExtensions{.eps,.pdf,.png,.jpg}
\else
    \DeclareGraphicsExtensions{.eps}
\fi
\usepackage{amsopn}
\DeclareMathOperator{\diag}{diag}
\usepackage{booktabs}

\newcommand{\TheTitle}{Optimal transfer operators for nonsymmetric two-grid methods}

\newcommand{\TheShortTitle}{Optimal transfer operators}

\allowdisplaybreaks

\DeclareMathOperator{\C}{\mathbb{C}}
\DeclareMathOperator{\N}{\mathbb{N}}

\DeclareMathOperator{\calR}{\mathcal{R}}
\DeclareMathOperator{\calN}{\mathcal{N}}

\newcommand{\abs}[1]{\ensuremath{\left\vert#1\right\vert}}
\newcommand{\m}[1]{\begin{bmatrix}#1\end{bmatrix}}
\newcommand{\Cnn}{\C^{n \times n}}
\newcommand{\Cnnc}{\C^{n \times n_c}}
\newcommand{\Cncnc}{\C^{n_c \times n_c}}
\newcommand{\Cn}{\C^n}

\newcommand{\Eplus}{E_+^{\nu_1, \nu_2}}
\newcommand{\Esouth}{E^{\nu_1, \nu_2}}
\newcommand{\tildeMinvB}{\widetilde{M}^{-1}B}
\newcommand{\hatMinvB}{\widehat{M}^{-1}B}
\newcommand{\lambdaquer}{\overline{\lambda}}

\DeclarePairedDelimiter{\norm}{\lVert}{\rVert}
\def\<{\langle}
\def\>{\rangle}
\def\{{\lbrace}
\def\}{\rbrace}

\newsiamremark{remark}{Remark}
\newsiamremark{example}{Example}

\author{Reinhard Nabben\thanks{Institute of Mathematics, Technische Universit\"at Berlin, (\email{nabben@math.tu-berlin.de}).} 
\and Ludwig Rooch\thanks{Institute of Mathematics, Technische Universit\"at Berlin, (\email{rooch@math.tu-berlin.de}).}}
\title{{\TheTitle}\thanks{Submitted to the editors DATE.}}
\headers{\TheShortTitle}{Reinhard Nabben, Ludwig Rooch}
\ifpdf
    \hypersetup{pdftitle={\TheTitle}, pdfauthor={Reinhard Nabben, Ludwig Rooch}}
\fi

\begin{document}

\maketitle

\begin{abstract}
  Algebraic Multigrid (AMG) methods have been proven to be effective solvers for large-scale linear algebraic systems $Ax = b$ with Hermitian positive definite (HPD) matrix $A$. For such problems the convergence in the $A$-norm is well understood, but for nonsymmetric indefinite systems fewer results exist. Recently, convergence results for more general $B$-norms induced by certain HPD matrices were established. There, orthogonal projections built by compatible transfer operators are used.
  
  Here, we present a theoretical framework for the convergence of nonsymmetric algebraic two-grid methods for arbitrary $B$-inner products and induced $B$-norms which naturally includes the HPD case and all recent results for the nonsymmetric case. For this purpose, we consider two different two-grid error operators with the first one being the natural generalization of the error operator in the HPD case. The second operator has been studied before and is simpler, but requires the additional assumption of normality in some inner product of the smoothing step $M^{-1}A$ to achieve convergence. We prove new convergence results, generalize some previous results and explain the differences and similarities of both operators together with the necessity of the normality. Moreover, we establish optimal compatible interpolation and restriction operators for both two-grid methods that minimize the error norm.
\end{abstract}



\section{Introduction}\label{sec:introduction}

We consider large-scale linear algebraic systems $Ax =b$ with nonsingular, nonsymmetric and indefinite system matrix $A$. For symmetric (Hermitian) positive definite (SPD or HPD) $A$, algebraic multigrid (AMG) methods have established themselves as efficient solvers and preconditioners and many theoretical results on the convergence, spectral properties and optimal transfer operators have been proven \cite{FalVas2004, FalVasZik2005, MacOls2014, Not2015, Vas2008}. There are several nonsymmetric and indefinite problems for which AMG methods have been developed and studied \cite{Lot2017, ManRugSou2018, ManMunRugSou2019, WieTumWalGee2014, BreManMcCRugSan2010} or spectral analysis has been done \cite{Not2010, MenNab2008a, MenNab2008b, Not2020, GarKehNab2020}. In \cite{Xu2022, BatNab2025, AliBraKahKrzSchSou2025, SouMan2024} optimal compatible transfer operators leading to orthogonal projections for nonsymmetric matrices were considered. 

However, a general framework for rigorous norm-based analysis remains to be found. For the simplified error propagation operator with coarse-grid correction and only one pre- or post-smoothing step, results for certain norm choices have been proven. These works include \cite{ManSou2019,BreManMcCRugSan2010,ManOlsSchSou2017} for the $\sqrt{A^HA}$-norm, \cite{Xu2022} for the $M$-norm (where $M^{-1}$ is the smoother) and \cite{BatNab2025} for arbitrary norms. In this article, we introduce a general framework for nonsymmetric algebraic two-grid methods that allows for the analysis of norm-based convergence in any norm given by some HPD matrix $B$. We consider two different error propagation operators, both with an arbitrary number of pre- and post-smoothing steps. The first operator is the natural generalization of the error operator in the HPD case. The second operator has a simpler structure and is easier to apply because it uses the same matrix as pre- and post-smoother. It was analyzed in \cite{KrzSouWimAliBraKah2025} where norm-based convergence results have been proven for a special class of norms. These norms are based on a diagonalization of $M^{-1}A$ which was also used to construct the optimal transfer operators. The diagonalizability of $M^{-1}A$ was first introduced in \cite{AliBraKahKrzSchSou2025} and is a new assumption in the convergence theory that has no equivalent in the existing HPD theory. This assumption actually refers to the concept of $B$-normal matrices which are the generalization of normal matrices in arbitrary $B$-inner products. We combine this with the existing idea of $B$-orthogonal projections \cite{BatNab2025, SouMan2024} to obtain our results. For both of our error operators we state the optimal compatible transfer operators. We discuss the differences of both operators and show that the structure of the first operator allows us to obtain the same convergence results as for the second operator but without any additional assumption on the smoothing matrix $M^{-1}A$. Our framework not only generalizes all existing results for the HPD theory, it also contains all herein mentioned results for nonsymmetric AMG.

In \cref{sec:proj and B-normal matrices} we recall some properties of projections and the concept of $B$-normal matrices. This is followed by the introduction of our two-grid setting in \cref{sec:B-normal in multigrid}. There, we define the two error propagation operators which we consider, discuss the $B$-orthogonality of the coarse-grid correction and introduce the symmetrized smoother matrices. Here, we already use the assumption of $B$-normality to achieve further results. Finally, in \cref{sec:two-grid conv} we prove the convergence of both error propagation matrices, state compatible transfer operators and show the optimality of these transfer operators. In a detailed discussion, we compare the two error operators and their respective assumptions for convergence. This is emphasized by some small theoretical  examples that demonstrate the differences of the two error operators.

Our theoretical framework is continued in \cite{NabRoo2026b}. There we give sharp bounds for both error  operators using $B$-orthogonal projections. We analyze the effect of the number of coarse variables on the norm of the error propagation matrices. Moreover, we extend McCormick's landmark $V$-cycle bound to nonsymmetric indefinite problems.

\section{Projections and \texorpdfstring{$B$}{B}-normal matrices} \label{sec:proj and B-normal matrices}

In the classical theory for AMG, the matrix $A$ is Hermitian positive definite (HPD) and the $A$-norm is used for any norm based estimate. Now, the matrix $A$ is nonsymmetric and indefinite in the standard inner product. The idea is to consider a different inner product in which we can symmetrize certain operators and matrices. For this purpose, we require the notion of a general $B$-inner product. In the following we always consider an HPD matrix $B \in \Cnn$.
The $B$-inner product and the induced $B$-(vector-)norm are defined as 
\begin{displaymath}
    \< x,y\>_B = y^H B x  \quad \text{and} \quad \lVert x \rVert_B = \sqrt{\<x,x\>_B}.
\end{displaymath}
Two vectors $x, y \in \mathbb{C}^n$ are called $B$-orthogonal (or orthogonal in the $B$-inner product) if $ \< x, y\>_B = 0 $. We then write $x \,\bot_B \, y$. We can also define the $B$-(matrix-)norm as operator norm of the $B$-(vector-)norm. For a matrix $C \in \Cnn$ the $B$-(matrix-)norm is given by
\begin{displaymath}
    \norm{C}_B = \sup_{x\neq 0} \frac{\norm{Cx}_B}{\norm{x}_B} = \norm{B^\frac12CB^{-\frac12}}_2.
\end{displaymath}

Furthermore, we denote the range or image of the matrix $C$ by $\calR(C)$ and the kernel of $C$ by $\calN(C)$. One main ingredient for any multigrid method are projections. 

\begin{definition}
    A matrix $\Pi \in \Cnn$ is called projection if $\Pi^2 = \Pi$. 
\end{definition}

We are specifically interested in the class of $B$-orthogonal projections, see also \cite[Definition~2.3]{Gal2004}.

\begin{definition} \label{def:B-orth proj}
  Let $\mathcal{U} \subseteq \mathbb{C}^n$. The unique operator $\Pi_{\mathcal{U}, \mathcal{U}^{\bot_B}}$ with $\Pi_{\mathcal{U}, \mathcal{U}^{\bot_B}}^2 = \Pi_{\mathcal{U}, \mathcal{U}^{\bot_B}}$ and
  $\mathcal{R}(\Pi_{\mathcal{U}, \mathcal{U}^{\bot_B}}) =\mathcal{\mathcal{U}}$ and $\mathcal{N}(\Pi_{\mathcal{U}, \mathcal{U}^{\bot_B}}) = \mathcal{\mathcal{U}^{\bot_B}}$ is called the $B$-orthogonal projection onto $\mathcal{U}$ along $\mathcal{U}^{\bot_B}$.
\end{definition}

Here, $\mathcal{U}^{\bot_B}$ denotes the $B$-orthogonal complement of $\mathcal U$. 

Note that we use the term $B$-orthogonal both for vectors and for projections. However, it is always clear from the context which concept is used. 
It follows from the definition that a projection $\Pi$ is $B$-orthogonal if and only if
\begin{displaymath}
    \mathcal{R}(\Pi)^{\bot_B} = \mathcal{N}(\Pi) \quad \text{or equivalently} \quad \mathcal{R}(\Pi) \, \bot_B \, \mathcal{N}(\Pi).
\end{displaymath}

One can also prove the following lemma (see \cite[Lemma~1]{SouMan2024}).

\begin{lemma}\label{lem:B-ortho proj}
    A projection $\Pi$ is $B$-orthogonal if and only if 
    \begin{equation}\label{eq:B-ortho proj}
        \Pi = B^{-1} \Pi^H B.
    \end{equation}
\end{lemma}

The following result (cf. \cite[Lemma~3]{SouMan2024} and \cite[Lemma~3.6]{Vas2008}) gives a useful characterization for the $B$-norm of a projection and its connection to $B$-orthogonality.

\begin{lemma} \label{lem:norm proj}
Let $\Pi \in \mathbb{C}^{n \times n}$ be a projection with $\Pi \neq 0$ and $\Pi \neq I_n$. Then it holds $\norm{\Pi}_B = \norm{I-\Pi}_B$ and
\begin{displaymath}
     \lVert \Pi \rVert^2_B = 1 + \sup_{x\in \mathcal{R}(\Pi)^{\bot_B}} \frac{\lVert \Pi x \rVert^2_B}{\lVert x \rVert^2_B}
\end{displaymath}
with $ \lVert \Pi \rVert^2_B = 1 $ if and only if $ \Pi $ is $B$-orthogonal.
\end{lemma}

Based on this idea of $B$-orthogonality, we introduce the general concept of the adjoint with respect to the $B$-inner product. 

\begin{definition}\label{def:B-adjoint B-normal B-ortho}
    Let $A \in \Cnn$. The $B$-adjoint of $A$ is defined as $A^+ = B^{-1}A^H B$. Then we say that $A$ is $B$-normal if $AA^+ = A^+A$ and $A$ is $B$-orthogonal if $A^+=A$. 
\end{definition}

The definition of $B$-adjoint is motivated by the fact that it is the unique matrix with
\begin{displaymath}
    \< Ax, y \>_B = \< x, A^+y \>_B
\end{displaymath}
for all $x,y \in \Cn$. For the standard inner product, that is $B= I_n$, the $I_n$-adjoint reduces to the Hermitian transpose and hence an $I_n$-orthogonal matrix is Hermitian. If $A$ is HPD and we consider the $A$-inner product, the $A$-adjoint is given by $A$. The definition of $B$-orthogonality in \cref{def:B-adjoint B-normal B-ortho} is consistent with the result of \cref{lem:B-ortho proj} because the right hand side of \cref{eq:B-ortho proj} is exactly the $B$-adjoint of $\Pi$. It is also easy to see that a $B$-orthogonal matrix $A$ is $B$-normal as well and satisfies $A^HB = BA $. 

The close relation of the $B$-adjoint to the Hermitian transpose also shows in its other properties.

\begin{proposition}
Let $A,C \in \Cnn$, then it holds $(A+C)^+ = A^+ + C^+$, $(AC)^+ = C^+A^+$ and $(A^+)^+ = A$.
\end{proposition}

\begin{remark}
    One can argue that it would be more suitable to say a matrix is $B$-Hermitian instead of $B$-orthogonal since the $B$-adjoint reduces to the Hermitian transpose for $B= I_n$. But our focus is on projections where the notion of $B$-orthogonal projections is already established and therefore we use the term $B$-orthogonal.
\end{remark}

Although $B$-orthogonality is a useful property, its assumption is very restricting. We will show that the weaker assumption of $B$-normality is often sufficient for our purposes. Similarly to classical normality there exist a lot of equivalent characterizations for $B$-normality. In \cite{ElsIkr1998} and \cite{GroJohSaWol1987} detailed lists with equivalent properties of $I_n$-normality can be found. Most of these properties can be generalized to $B$-normal matrices. We state a few selected characterizations that are relevant for this work.

\begin{definition}
    $A$ is called $B$-unitary if $A^H B A = I_n$, i.e. the columns of $A$ are orthonormal in the $B$-inner product.
\end{definition}

For a $B$-unitary matrix $A$ it holds $A^+A = B^{-1}$ and $A^{-1} = A^H B$ as well as $B = A^{-H}A^{-1}$. The last equation already shows that the structure of the matrix $B$ also depends on the underlying matrix $A$. This behavior is described in the fifth characterization for $B$-normality in the following result which is mostly proven in \cite[Theorem~3.1]{LieStr2008}.   

\begin{theorem}\label{thm:B-normal_charact}
    The following statements are equivalent: 
    \begin{enumerate}[(1)]
        \item $A$ is $B$-normal.
        \item \label{thm:enum:Aplus p(A)} $A^+ = p(A)$ for some polynomial $p \in \C[z]$.
        \item \label{thm:enum:B-unitarily diag} $A$ is $B$-unitarily diagonalizable, that is there exists a $B$-unitary matrix $U \in \Cnn$ and a diagonal matrix $D\in \Cnn$ such that $A = U D U^{-1}$. 
        \item \label{thm:enum:eigenvector A Aplus} If $x$ is an eigenvector of $A$ to the eigenvalue $\lambda$, then $x$ is an eigenvector of $A^+$ to the eigenvalue $\lambdaquer$.
        \item \label{thm:enum:diag and decomp B} $A$ is diagonalizable with eigendecomposition $A = W\Lambda W^{-1}$ (without loss of generality we consider the eigenvalues and eigenvectors of $A$ ordered so that equal eigenvalues form a single diagonal block in $\Lambda$). Furthermore, using the eigenvector matrix $W$, the matrix $B^{-1}$ has a decomposition $B^{-1} = WDW^H$ where $D$ is an HPD block diagonal matrix with block sizes corresponding to those in $\Lambda$. There also exists a polynomial $p \in \C[z]$ such that $p(\Lambda) = \Lambda^H$. 
    \end{enumerate}
\end{theorem}
\begin{proof}
    The properties follow from the fact that $AA^+ = A^+A$ is equivalent to 
    \begin{displaymath}
        (B^{\frac12} AB^{-\frac12}) (B^{-\frac12}A^H B^{\frac12}) = (B^{-\frac12}A^H B^{\frac12})(B^{\frac12}A B^{-\frac12})
    \end{displaymath} 
    and hence $\widetilde{A} = B^{\frac12}A B^{-\frac12}$ is normal. Therefore $B$-normality can be reduced to normality in the classic sense and any characterization from \cite{GroJohSaWol1987} and \cite{ElsIkr1998} can be used.

    $\emph{(1)} \Leftrightarrow \emph{(3)}$: This proof is based on the proof of \cite[Theorem~3.1]{LieStr2008}. The matrix $\widetilde{A}$ can be unitarily diagonalized, i.e. $\widetilde{A} = U\Lambda U^H$ hence $A = B^{-\frac12} U \Lambda U^H B^{\frac12}$. With $W := B^{-\frac12} U$ we get $A = W\Lambda W^{-1}$ with $W^HBW = I$ (in particular $W^{-1} = W^HB$). Thus $A$ is $B$-unitarily diagonalizable. Conversely, if $A$ is $B$-unitary diagonalizable we have $A = U\Lambda U^{-1}$ with $U^HBU = I$. Then it holds
    \begin{align*}
        AA^+ &= U\Lambda U^{-1}B^{-1}U^{-H}\Lambda^H U^HB = U\Lambda \Lambda^H U^{-1} = U \Lambda^H \Lambda U^{-1} \\
        &= B^{-1}U^{-H}\Lambda^HU^H BU \Lambda U^{-1} = A^+A
    \end{align*}
    and thus $A$ is $B$-normal. We used that diagonal matrices commute and $U^HB = U^{-1}$.

    $\emph{(1)} \Leftrightarrow \emph{(2)}$: Let $A$ be $B$-normal. By $\emph{(3)}$, $A$ is $B$-unitarily diagonalizable with $A = U\Lambda U^{-1}$ where $U$ is $B$-unitary and $\Lambda$ is a diagonal matrix containing the eigenvalues of $A$. Furthermore, there exists an interpolation polynomial with $p(\lambda) = \overline{\lambda}$ for all eigenvalues $\lambda$ of $A$, and thus 
    \begin{displaymath}
        A^+ = B^{-1} A^H B = B^{-1} U^{-H} \Lambda^H U^H B = B^{-1} BU p(\Lambda) U^{-1} = U p(\Lambda) U^{-1} = p(A).
    \end{displaymath}
    On the other hand, if $A^+ = p(A)$, then $A^+$ commutes with $A$.

    $\emph{(1)} \Leftrightarrow \emph{(4)}$: Let $A$ be $B$-normal and $(\lambda, x)$ be an eigenpair of $A$, that is $Ax = \lambda x$. By $(2)$ we have $A^+x = p(A)x = p(\lambda)x$, hence $x$ is an eigenvector of $A^+$. We even have
    \begin{displaymath}
        p(\lambda) \<x,x\>_B = \<p(\lambda)x,x\>_B = \<A^+x,x\>_B = \<x,Ax\>_B = \overline{\lambda} \< x, x \>_B
    \end{displaymath}
    and because $x \neq 0$ it holds $p(\lambda) = \overline{\lambda}$. Therefore, $x$ is an eigenvector of $A^+$ to the eigenvalue $\overline{\lambda}$.
    
    Conversely, if $\emph{(4)}$ holds then $y := B^\frac12 x$ is an eigenvector of $\widetilde{A}$ and of $\widetilde{A}^H$. Therefore $\widetilde{A}$ is normal and thus $A$ is $B$-normal (cf. \cite[Condition~10]{GroJohSaWol1987}).

    $\emph{(1)} \Leftrightarrow \emph{(5)}$: This is an immediate consequence of \cite[Theorem~3.1]{LieStr2008}.
\end{proof}

The concept of $B$-normal matrices is also used in the theory of short recurrences for computing $B$-orthogonal Krylov subspace bases \cite{LieStr2008, FabMan1984, LieStr2013}. There, the length of the recurrence for the computation of the Krylov subspace basis vectors depends on the degree of the polynomial $p$ in \cref{thm:B-normal_charact}~(\labelcref{thm:enum:Aplus p(A)}). Note that we are not interested in the degree of $p$ at all but rather the set $\mathcal{B}$ of all possible matrices $B$ such that $A$ can be $B$-normal. This set can be characterized using equivalence \cref{thm:B-normal_charact}~(\labelcref{thm:enum:diag and decomp B}) and is given by
\begin{align}\label{eq:setB}
    \mathcal{B} = \{ (WDW^H)^{-1} \mid \ &D \text{ is an HPD block diagonal matrix} \\ 
    &\text{with block sizes corresponding to those in } \Lambda \}. \nonumber
\end{align}

$\Lambda$ denotes the eigenvalue matrix of $A$. The structure of the matrices in $\mathcal{B}$ will be of importance later when we derive convergence bounds for AMG methods. Using these characterizations, we can conclude some useful properties of $B$-orthogonal matrices.

\begin{proposition} \label{prop:eigB-ortho}
    Let $A \in \Cnn $ be $B$-orthogonal. Then all eigenvalues of $A$ are real.
\end{proposition}
\begin{proof}
    Since $A$ is also $B$-normal, there exists a polynomial $p$ with $p(A) = A^+ = A$, hence $p(z) = z$ for all $z \in \C$. The eigenvalues of $A$ then satisfy $p(\lambda) = \overline{\lambda} = \lambda$ and thus all eigenvalues of $A$ are real. 
\end{proof} 

\begin{proposition}\label{prop:B-orthogonal_iff_eig_real}
    Let $A \in \Cnn$ be $B$-normal. Then $A$ is $B$-orthogonal if and only if all eigenvalues of $A$ are real.
\end{proposition}
\begin{proof}
     Let $A$ be $B$-normal. The first part is proven in \cref{prop:eigB-ortho}. Assume that the eigenvalues $\lambda$ of $A$ are real. Then it holds $\overline{\lambda} = \lambda = p(\lambda)$ for the identity polynomial $p$. Since $A$ is $B$-normal, we know by \cref{thm:B-normal_charact}~(\labelcref{thm:enum:diag and decomp B}) that $A$ has an eigendecomposition $A = W\Lambda W^{-1}$ and $B^{-1} = WDW^H$ for an HPD block diagonal matrix $D$ where the block sizes corresponds to those in $\Lambda$. Then it holds
    \begin{align*}
        A^+ &= B^{-1}A^HB = (WDW^H)(W^{-H}\Lambda^H W^H)(W^{-H}D^{-1}W^{-1})
        \\& = WD\Lambda^HD^{-1}W^{-1} = W \Lambda^H W^{-1} = W\Lambda W^{-1} = A
    \end{align*}
    and thus $A$ is $B$-orthogonal. Here, $D$ and $\Lambda$ commute since they have the same block structure
\end{proof}

Using this result, we obtain another important property of the $B$-norm which for $B=I_n$ again reduces to a known property of the spectral norm.

\begin{proposition}\label{prop:B-norm lambda_max}
    Let $A \in \Cnn$. Then it holds
    \begin{displaymath}
        \norm{A}_B^2  = \norm{AA^+}_B = \norm{A^+A}_B = \norm{A^+}_B^2 = \lambda_{max}(A^+A) = \lambda_{max}(AA^+),
    \end{displaymath}
    where $\lambda_{max}$ denotes the largest eigenvalue.
\end{proposition}

\begin{proof}
    Since all the properties are known for $B = I_n$, that is the spectral norm, we can use the fact that $\norm{A}_B = \norm{B^{\frac12} A B^{-\frac12}}_2$ to obtain the result.
\end{proof}

\begin{remark}
    Similar to $\lambda_{max}$ we define $\lambda_{min}$ as the smallest eigenvalue of a matrix. Both of these quantities are well-defined if the matrix only has real eigenvalues. But since $A^+A$ is $B$-orthogonal for any matrix $A \in \Cnn$, by \cref{prop:eigB-ortho} its eigenvalues are always real.
\end{remark}

Another important property of $B$-normal matrices is their relation to the spectral radius $\rho$. 
\begin{proposition}\label{prop:B-normal_spectral_radius}
    Let $A \in \Cnn$ be $B$-normal. Then $\norm{A}_B = \rho(A)$. 
\end{proposition}

\begin{proof}
    With \cref{thm:B-normal_charact}~(\labelcref{thm:enum:B-unitarily diag}) the matrix $A$ is $B$-unitarily diagonalizable, that is $A = W \Lambda W^{-1}$ with $W^HBW = I_n$ and $\Lambda = \diag(\lambda_1, \dots, \lambda_n)$ containing the eigenvalues of $A$. Then we have
    \begin{align*}
        \norm{A}_B^2 &= \lambda_{max}(A^+A) = \lambda_{max}(B^{-1}W^{-H}\Lambda^HW^HBW\Lambda W^{-1}) = \lambda_{max}(W\Lambda^H\Lambda W^{-1})
        \\ & = \lambda_{max}(\Lambda^H\Lambda) = \lambda_{max}(\diag(\abs{\lambda_1}^2, \dots, \abs{\lambda_n}^2))
        = \max_{i=1, \dots, n} \abs{\lambda_i}^2 = \rho(A)^2.
    \end{align*}
    Taking the square root yields the result.
\end{proof}

\section{\texorpdfstring{$B$}{B}-normal matrices in nonsymmetric AMG}\label{sec:B-normal in multigrid}
We consider the linear algebraic system
\begin{displaymath}
    Ax = b
\end{displaymath}
where $A \in \Cnn$ is a nonsingular, nonsymmetric and indefinite matrix and $b \in \C^n$. We want to solve this system using algebraic two-grid methods. Such methods typically consist of two components, the smoothing and the coarse-grid correction. Smoothing is usually a basic iterative method such as the Gauss-Seidel or Jacobi method and is done for a few iterations. The coarse-grid correction is based on a projection and requires the division of the $n$ variables in two smaller groups with $n_c < n$ and $n_f < n$ variables such that $n_c + n_f = n$. The number $n_c$ can be interpreted as the number of variables of the coarse grid although in AMG methods there is no real underlying geometric grid. However, to pass from $n$ to $n_c$ variables and vice versa, a restriction matrix $R \in \Cnnc$ and an interpolation matrix $P \in \Cnnc$ are needed. The coarse grid matrix is then given by
\begin{displaymath}
    A_c = R^HAP \in \Cncnc
\end{displaymath}
which we always assume to be nonsingular. This implies that $R$ and $P$ must have full rank. For handling singular coarse grid system see e.g. \cite{Not2016, KehNab2016}. For HPD $A$, the typical choice is $R = P$ with full rank which always yields nonsingular $A_c$.

A typical approach to guarantee the convergence of a two-grid method is to assume some kind of smoothing and approximation property of the smoothing step and the coarse-grid correction, respectively. We follow this strategy in order to establish $B$-norm bounds for the algebraic two-grid error operator. 

\subsection{The coarse-grid correction}

We start the analysis with the coarse-grid correction $\Pi_A$ where 
\begin{displaymath}
    \Pi_A := \Pi_A(P,R) := P(R^HAP)^{-1}R^HA = PA_c^{-1}R^HA \in \Cnn
\end{displaymath}
is a projection onto $\calR(P)$ along $\calN(R^HA)$. Here, we write $\Pi_A(P,R)$ to denote the dependency on $P$ and $R$. 

A key property for the following proofs is the $B$-orthogonality of the coarse-grid correction. This property can be characterized by multiple equivalent conditions, some of which we have already introduced in \cref{sec:proj and B-normal matrices}. The following result summarizes some of them based on the \cref{lem:B-ortho proj,lem:norm proj} and \cite[Corollary~2]{SouMan2024}.

\begin{lemma}\label{lem:equiv_cond Pi B-ortho}
    Let $R,P \in \Cnnc$ such that $R^HAP$ be nonsingular. Then $\Pi_A$ is a projection and the following statements are equivalent:
    \begin{enumerate}[(1)]
        \item $\Pi_A$ is $B$-orthogonal.
        \item $\Pi_A^+ = \Pi_A$.
        \item $\calR(\Pi)^{\bot_B} = \calN(\Pi_A)$.
        \item $\norm{\Pi_A}_B = \norm{I-\Pi_A}_B = 1$.
        \item $\calR(BP) = \calR(A^HR)$.
        \item \label{lem:enum:Pi B-ortho equal ranges BatNab} $\calR(P) = \calR(B^{-1}A^HR)$.
        \item $\calN(R^HA) = \calN(P^HB)$. 
    \end{enumerate}
\end{lemma}

In \cite{BatNab2025} condition (\labelcref{lem:enum:Pi B-ortho equal ranges BatNab}) is used to obtain a $B$-orthogonal coarse-grid correction. There, given a restriction $R \in \Cnnc$ of full rank, the corresponding interpolation that yields a $B$-orthogonal coarse-grid correction is defined as $P_* = B^{-1}A^HR$. Conversely, given $P \in \Cnnc$, the corresponding restriction is defined as $R_* = A^{-H}BP$. Both operators require the inverse of either $A^H$ or $B$ which can be problematic. Examples for $P_*$ and $R_*$ for special cases of $B$ are discussed in \cite{BatNab2025}. However, when using $P_*$ or $R_*$ we obtain
\begin{equation}\label{eq:Pi_A with Pstar}
    \Pi_A(P_*,R) = \Pi_A(P,R_*) = P_*(P^H_* BP_*)^{-1}P_*^HB = \Pi_B(P_*,P_*).
\end{equation}
Choosing $P_*$ symmetrized the coarse-grid correction in the sense that $\Pi_B(P_*,P_*)$ has the same structure as $\Pi_A(P,P)$ in the HPD case. This is no coincidence as for $A=B$ HPD we have $P_* = R$ and $\Pi_B(P_*,P_*)$ becomes $\Pi_A(P,P)$. Note that any interpolation $P$ or restriction $R$ can be multiplied by any nonsingular matrix $S \in \Cncnc$ from the right without changing the resulting coarse-grid correction. The spaces $\calR(P)$ and $\calR(R)$ are independent of their choice of bases. Hence, any two operators $P_1, P_2$ or $R_1, R_2$ with $\calR(P_1) = \calR(P_2)$ or $\calR(R_1) = \calR(R_2)$ produce the same coarse-grid correction. This is especially true for $P_*$ and $R_*$ which already shows that there does not exist a unique pair of operators $(P,R)$ that results in a $B$-orthogonal coarse-grid projection.

The main advantage of using $P_*$ or $R_*$ is the fact that $\Pi_A(P_*,R)$ and $\Pi_A(P,R_*)$ are $B$-orthogonal for any HPD matrix $B$. This does not contradict \cref{thm:B-normal_charact}~(\labelcref{thm:enum:diag and decomp B}) and \cref{eq:setB} where all possible HPD matrices $B$ such that $\Pi_A$ can be $B$-orthogonal are characterized, because changes in $B$ result in different operators $P_*$ and $R_*$ and thus different $\Pi_A$. Note, that $\Pi_A$ is $B$-orthogonal if and only if $I - \Pi_A$ is $B$-orthogonal.

\subsection{The iteration matrices}

The algebraic two-grid methods and its error propagation matrix consist of pre- and post-smoothing steps combined with the coarse-grid correction. In our framework we consider the two error propagation operators
\begin{align}
    \Eplus := \Eplus(P,R) &:= (I-(M^{-1}A)^+)^{\nu_2}(I-\Pi_A(P,R))(I-M^{-1}A)^{\nu_1}, \label{eq:Eplus}
    \\ \Esouth := \Esouth(P,R) &:= (I-M^{-1}A)^{\nu_2}(I-\Pi_A(P,R))(I-M^{-1}A)^{\nu_1}.\label{eq:Esouth}
\end{align}
If the dependency of $R$ and $P$ is clear from the context, we simply write $\Eplus$ and $\Esouth$. The exponents $\nu_1, \nu_2 \in \N_0$ are the number of pre- and post-smoothing steps, respectively, and $M^{-1} \in \Cnn$ is the smoother which we always assume to be nonsingular. An obvious differences between the two operators is the $B$-adjoint of $M^{-1}A$ in the post-smoothing step of $\Eplus$ which introduces the additional dependency of $B$ into the operator. 

Since $\Pi_A$ is a projection, we can always split up both error propagation operators 
\begin{equation}
\Eplus = E_+^{0,\nu_2} \cdot E_+^{\nu_1,0} \quad \text{and} \quad \Esouth = E^{0,\nu_2} \cdot E^{\nu_1,0}.
\end{equation}

Although the operator $\Eplus$ is more complicated than $\Esouth$, it is the natural generalization of the error propagation operator for Hermitian positive definite $A$ which is defined as
\begin{equation}\label{eq:error_op HPD case}
    E_{\text{HPD}}^{\nu_1, \nu_2}(P,P) := (I-M^{-H}A)^{\nu_2}(I-P(P^HAP)^{-1}P^HA)(I-M^{-1}A)^{\nu_1}.
\end{equation}
For $A = B$ HPD the $A$-adjoint of $M^{-1}A$ is given by
\begin{displaymath}
    (M^{-1}A)^+ = A^{-1}(M^{-1}A)^HA = M^{-H}A
\end{displaymath}
and with $R = P$ the operator $\Eplus(P,P)$ is exactly the operator $E_{\text{HPD}}^{\nu_1, \nu_2}(P,P)$. 

On the other hand, there is the simpler error operator $\Esouth$ that uses the same smoother for pre- and post-smoothing, making the whole operator easier to apply. It was studied in \cite{Not2010, Not2016} and recently in \cite{KrzSouWimAliBraKah2025}. For the simple case $\nu_1 = 1$ and $\nu_2 = 0$, it holds $E_+^{1,0} = E^{1,0}$. This operator has been analyzed in \cite{AliBraKahKrzSchSou2025, BatNab2025, Xu2022}. 

If $M^{-1}A$ is $B$-orthogonal, it holds $\Eplus(P,R) = \Esouth(P,R)$ and the error matrices coincide. However, in general the two operators differ from each other and therefore we analyze their convergence behavior separately. 

The next result illustrates the fact that $\Eplus$ is the natural generalization of the error operator in the HPD case. It proves properties of $\Eplus$ that reduce to known properties of $E_{\text{HPD}}^{\nu_1, \nu_2}$ if we consider $A=B$ to be HPD \cite{FalVas2004}. We also see that $\Esouth$ requires additional assumptions to achieve the same results.

\begin{theorem} \label{thm:properties Eplus Esouth}
    Let $\Pi_A$ be $B$-orthogonal and $\nu \in \N_0$. Then it holds 
    \begin{enumerate}[1)] 
        \item $(E_+^{\nu,\nu})^+ = E_+^{\nu,\nu}$.
        \item \label{thm:enum:property Eplus only one smoothing adjoint} $(E_+^{\nu,0})^+ = E_+^{0,\nu}$ and $(E_+^{0,\nu})^+ = E_+^{\nu,0}$
        \item $B^{-1} (E_+^{\nu,0})^H = E_+^{0,\nu}B^{-1}$ and $B^{-1}(E_+^{0,\nu})^H = E_+^{\nu,0}B^{-1}$.
        \item \label{thm:enum:property norm Eplus split} $\norm{E_+^{\nu,\nu}}_B = \norm{E_+^{\nu,0}}_B^2 = \norm{E_+^{0,\nu}}_B^2$.
   \end{enumerate}
    If additionally $M^{-1}A$ is $B$-orthogonal, then the above properties also hold for $E^{\nu,\nu}$.
\end{theorem}
\begin{proof}
    Let $\Pi_A$ be $B$-orthogonal. \emph{1)} and \emph{2)} both follow from the properties of the $B$-adjoint. Since $\Pi_A$ is $B$-orthogonal, it holds $(I-\Pi_A)^+ = (I-\Pi_A)$ and thus
    \begin{align*}
        (E_+^{\nu,\nu})^+ &= \left((I-(M^{-1}A)^+)^\nu(I-\Pi_A)(I-M^{-1}A)^\nu \right)^+ 
        \\ &= (I-(M^{-1}A)^+)^\nu(I-\Pi_A)^+(I-M^{-1}A)^\nu = E_+^{\nu,\nu}
    \end{align*}
    and
    \begin{align*}
        (E_+^{\nu,0})^+ &= \left((I-\Pi_A)(I-M^{-1}A)^\nu \right)^+ = (I-(M^{-1}A)^+)^\nu(I-\Pi_A^+) = E_+^{0,\nu}.
    \end{align*}
    Taking the $B$-adjoint on both sides yields the second equation of \emph{2)}. \emph{3)} is then a reformulation of \emph{2)}. And for \emph{4)} we use \emph{2)} and the fact that $E_+^{\nu,\nu} = E_+^{0,\nu} \cdot E_+^{\nu,0}$ to obtain
    \begin{displaymath}
        \norm{E_+^{\nu,\nu}}_B = \norm{E_+^{0,\nu} E_+^{\nu,0}}_B = \norm{E_+^{0,\nu} (E_+^{0,\nu})^+}_B = \norm{E_+^{0,\nu}}_B^2
    \end{displaymath}
    with \cref{prop:B-norm lambda_max}. An analogous calculation shows $\norm{E_+^{\nu,\nu}}_B = \norm{E_+^{\nu,0}}_B^2$. If additionally $M^{-1}A$ is $B$-orthogonal, the operators $E_+^{\nu,\nu}$ and $E^{\nu,\nu}$ coincide and thus \emph{1)} to \emph{4)} also hold for $E^{\nu,\nu}$. 
\end{proof}

\subsection{Smoothing and generalized eigenvalue  problems}

Our convergence analysis and in particular our optimality results depend on eigenvectors and eigenvalues of certain (generalized) eigenvalue problems. The first eigenvalue problem is given by
\begin{equation}\label{eq:eig AM}
    M^{-1}Az = \lambda z. 
\end{equation} 
For the second eigenvalue problem we define the following matrices
\begin{align}
    \widetilde{M}^{-1} &= M^{-1}AB^{-1} + B^{-1}A^HM^{-H} - B^{-1}A^HM^{-H}BM^{-1}AB^{-1}, \label{eq:tildeMinv}
    \\ \widehat{M}^{-1} &= M^{-1}AB^{-1} + B^{-1}A^HM^{-H}- M^{-1}AB^{-1}A^HM^{-H} \label{eq:hatMinv}
\end{align}
which satisfy
\begin{align}
    \widetilde M^{-1}B &= M^{-1}A + (M^{-1}A)^+ - (M^{-1}A)^+ M^{-1}A, \label{eq:tildeMinvB}
    \\ \hatMinvB &=  M^{-1}A + (M^{-1}A)^+ - M^{-1}A (M^{-1}A)^+ \label{eq:hatMinvB}
\end{align}
and thus
\begin{align*}
    I-\tildeMinvB &= (I-M^{-1}A)^+(I-M^{-1}A),
    \\ I- \hatMinvB &= (I-M^{-1}A)(I-M^{-1}A)^+.
\end{align*}
Note that  $\widetilde{M}^{-1} $ and $\widehat{M}^{-1} $ are Hermitian. Then we consider the  second eigenvalue problem 
\begin{equation}\label{eq:eig hatMinvB}
    \hatMinvB z = \mu z.
\end{equation}
Upon first sight $\widehat M^{-1}$ and $\widetilde M^{-1}$ look complicated but their structure becomes apparent for $B=A$ HPD. Then 
\begin{align*}
    \widetilde M^{-1} &= M^{-1} + M^{-H} - M^{-H}AM^{-1} = M^{-H}(M + M^H - A)M^{-1}, \\
    \widehat M^{-1} &= M^{-1} + M^{-H} - M^{-H}AM^{-1} = M^{-1}(M + M^H - A)M^{-H}
\end{align*}
and 
\begin{align*}
    \widetilde M &=  M(M+M^H-A)^{-1}M^H, \\
    \widehat M &= M^H(M+M^H-A)^{-1}M.
    \end{align*}
These matrices appear quite often in the AMG analysis of the HPD case, e.g. as symmetrized smoother, see \cite{XuZik2017, Vas2008, MacOls2014}. Moreover, they are used to determine optimal restriction and prolongation operators \cite{BraCaoKahFalHu2018, XuZik2017, GarNab2019}.

The matrices $\tildeMinvB$ and $\hatMinvB$ have the same structure and differ only in the order of the last term. Their properties are summarized in the following result.
\begin{proposition} \label{prop:properties tildeMinvB hatMinvB}
    Let $\widetilde{M}^{-1} $ and $\widehat{M}^{-1} $ be defined as in \cref{eq:tildeMinv} and \cref{eq:hatMinv}. Then $\tildeMinvB$ and $\hatMinvB$ are $B$-orthogonal and have real eigenvalues. 
    If $M^{-1}A$ is additionally $B$-normal it holds $\widetilde M^{-1} = \widehat M^{-1}$ and $\tildeMinvB = \hatMinvB$. 
\end{proposition}
\begin{proof}
    It is obvious that 
    \begin{displaymath}
        (\tildeMinvB)^+ = \tildeMinvB \quad \text{and} \quad (\hatMinvB)^+ = \hatMinvB,
    \end{displaymath}
    hence $\hatMinvB$ and $\tildeMinvB$ are $B$-orthogonal. Then by \cref{prop:eigB-ortho} their eigenvalues are real. If $M^{-1}A$ is additionally $B$-normal, it commutes with $(M^{-1}A)^+$ and we have $\tildeMinvB = \hatMinvB$ and therefore also $\widetilde M^{-1} = \widehat M^{-1}$.
\end{proof}

Clearly, if $M^{-1}A$ is $B$-orthogonal, it also holds $\widetilde M^{-1} = \widehat M^{-1}$, but \cref{prop:properties tildeMinvB hatMinvB} shows that $B$-normality is already sufficient to obtain this equality. 

One goal of theoretical two-grid analysis is to estimate the convergence speed of the method by bounding the norm of the error propagation matrix. If $A$ is HPD, the convergence is typically measured in the $A$-norm, also known as the energy norm. Here, $A$ is in general nonsymmetric and indefinite and therefore we use an arbitrary $B$-norm. A simple approach to estimate such a $B$-norm is to investigate and bound the smoothing steps and the coarse-grid correction individually. By the consistency of the $B$-norm it holds
\begin{displaymath}
    \norm{\Esouth}_B \leq \norm{I-M^{-1}A}_B^{\nu_1 + \nu_2} \norm{I-\Pi_A}_B
\end{displaymath}
and for $\norm{\Eplus}_B$ analogously. A two-grid method is called convergent if the norm of the corresponding error operator is less than one, that is $\norm{\Esouth}_B <1$ or $\norm{\Eplus}_B <1$. 

By \cref{lem:norm proj} we already know that $\norm{I-\Pi_A}_B$ attains its minimal value one if and only if $\Pi_A$ is $B$-orthogonal. Therefore, one of our goals is to find restriction and interpolation operators $R$ and $P$ such that $\Pi_A(P,R)$ is $B$-orthogonal for some HPD $B$. 

Besides the $B$-orthogonality of the coarse-grid correction, we require some additional assumption for the smoothing to obtain convergence results. This condition is given by the usual smoothing assumption but now in the $B$-norm.
\begin{definition}
    We say that the smoothing assumption is satisfied if 
    \begin{displaymath}
        \norm{I-M^{-1}A}_B < 1.
    \end{displaymath}
\end{definition}
We can characterize this property with the following theorem which was introduced and partly proven in \cite{BatNab2025}.

\begin{theorem}\label{thm:smoothing_assumpt_BatNab}
Let $\widetilde M^{-1}$ and $\widehat M^{-1}$ be defined as in \cref{eq:tildeMinv} and \cref{eq:hatMinv}, respectively. Then the following statements are equivalent:
\begin{enumerate}[(1)]
    \item $\lVert I - M^{-1}A\rVert_B < 1$.
    \item $\widetilde{M}^{-1}$ is HPD.
    \item $\sigma(\widetilde{M}^{-1}B) \subseteq (0,1] $.
    \item $\widehat M^{-1}$ is HPD.
    \item $\sigma(\hatMinvB) \subseteq (0,1]$.
\end{enumerate}
\end{theorem}

\begin{proof}
    The proof in \cite{BatNab2025} shows that the first two statements are equivalent. 
    
    $\emph{(1)} \Rightarrow \emph{(3)}$: Let $\norm{I-M^{-1}A}_B < 1$. Then $\widetilde{M}^{-1}$ is HPD and so is $B^{\frac12}\widetilde{M}^{-1} B^{\frac12}$. By similarity we get $\sigma(\widetilde{M}^{-1}B) \subseteq (0, \infty)$. Further, it holds
    \begin{align}\label{eq:lambda max I-tildeMinvB}
        \lambda_{max}(I-B^{\frac12}\widetilde{M}^{-1}B^{\frac12}) &= \lambda_{max}((I-B^{\frac12}M^{-1}AB^{-\frac12})^H(I-B^{\frac12}M^{-1}AB^{-\frac12})) \nonumber
        \\ &= \lambda_{max}((I-M^{-1}A)^+(I-M^{-1}A)) = \norm{I-M^{-1}A}_B^2 < 1.
    \end{align}
    But, due to its structure the matrix $I-B^{\frac12}\widetilde{M}^{-1}B^{\frac12}$ is Hermitian positive semidefinite (HPSD) and thus has real nonnegative eigenvalues. Therefore, it holds $\sigma(I-B^{\frac12}\widetilde{M}^{-1}B^{\frac12}) \subseteq [0,\infty)$ and hence $\sigma(\hatMinvB) = \sigma(B^{\frac12}\widetilde{M}^{-1}B^{\frac12}) \subseteq (0,1]$.

    $\emph{(3)} \Rightarrow \emph{(1)}$: Let $\sigma(\hatMinvB) \subseteq (0,1]$. Then $\sigma(I-B^{\frac12}\widetilde{M}^{-1}B^{\frac12}) = \sigma(I-\tildeMinvB) \subseteq [0,1) $. With \cref{eq:lambda max I-tildeMinvB} we obtain $\norm{I-M^{-1}A}_B^2 = \lambda_{max}(I-B^{\frac12}\widetilde{M}^{-1}B^{\frac12}) < 1$.

    $\emph{(1)} \Rightarrow \emph{(5)}$: The proof is analogous to $\emph{(1)} \Rightarrow \emph{(3)}$ and follows by using 
    \begin{displaymath}
        \lambda_{max}((I-M^{-1}A)^+(I-M^{-1}A)) = \lambda_{max}((I-M^{-1}A)(I-M^{-1}A)^+)
    \end{displaymath}
    and replacing $\widetilde M^{-1}$ with $\widehat M^{-1}$.

    $\emph{(5)} \Rightarrow \emph{(4)}$: Let $\sigma(\hatMinvB) \subseteq (0,1]$. By its structure 
    \begin{displaymath}
        B^{\frac12}\widehat{M}^{-1} B^{\frac12} = (I-B^{\frac12}M^{-1}AB^{-\frac12})(I-B^{\frac12}M^{-1}AB^{-\frac12})^H
    \end{displaymath}
    is HPSD and with $\sigma(B^{\frac12}\widehat{M}^{-1} B^{\frac12}) = \sigma(\hatMinvB) \subseteq (0,1]$ it must be HPD. Then $\widehat M^{-1}$ is also HPD.

    $\emph{(4)} \Rightarrow \emph{(1)}$: Assume that $\widehat M^{-1}$ is HPD. Then $B^{\frac12}\widehat{M}^{-1} B^{\frac12}$ is also HPD and thus has positive real eigenvalues. Then every eigenvalue of $I - B^{\frac12}\widehat{M}^{-1} B^{\frac12}$ must be smaller than $1$ and with a similar argument as \cref{eq:lambda max I-tildeMinvB} we obtain $\emph{(1)}$.
\end{proof}

The smoothing assumption can be characterized even further if we add the assumption that $M^{-1}A$ is $B$-normal.

\begin{corollary}\label{cor:equiv_smoothing_assump}
    Let $\widetilde M^{-1}$ and $\widehat M^{-1}$ be defined as in \cref{eq:tildeMinv} and \cref{eq:hatMinv}, respectively. Let $M^{-1}A$ be $B$-normal. Then the following statements are equivalent:
    \begin{enumerate}[(1)]
        \item $\lVert I - M^{-1}A\rVert_B < 1$.
        \item $\widetilde{M}^{-1} = \widehat M^{-1}$ is HPD.
        \item $\sigma(\widetilde{M}^{-1}B) = \sigma(\hatMinvB) \subseteq (0,1] $.
        \item $\abs{\lambda - 1}^2 < 1$ for all eigenvalues $\lambda$ of $M^{-1}A$.
        \item $\rho(I-M^{-1}A) < 1$
    \end{enumerate}
\end{corollary} 
\begin{proof}
    Let $M^{-1}A$ be $B$-normal. Then $\widehat M^{-1}$ and $\widetilde M^{-1}$ coincide (see \cref{prop:properties tildeMinvB hatMinvB}) and the first three equivalences result from \cref{thm:smoothing_assumpt_BatNab} above. 
    
    $\emph{(3)} \Leftrightarrow \emph{(4)}$: Let $(\lambda, z)$ be an eigenpair of $M^{-1}A$. By \cref{thm:B-normal_charact}~(\labelcref{thm:enum:eigenvector A Aplus}), the eigenvector $z$ of $M^{-1}A$ to the eigenvalue $\lambda$ is also an eigenvector of $(M^{-1}A)^+$ to the eigenvalue $\overline{\lambda}$. Thus, we obtain
    \begin{align*}
        \tildeMinvB z &= (M^{-1}A + (M^{-1}A)^+ - (M^{-1}A)^+M^{-1}A)z = (\lambda + \lambdaquer - \lambdaquer \lambda) z
    \end{align*}
    with $0 < \lambda + \lambdaquer - \lambdaquer \lambda \leq 1$. By adding $-1$ on both sides and multiplying by $-1$ we get
    \begin{displaymath}
        0 \leq 1 - \lambda - \lambdaquer + \lambdaquer \lambda < 1
    \end{displaymath}
    which is equivalent to $0 \leq \abs{\lambda - 1}^2 < 1$. Because $\lambda$ was arbitrary, this holds for all eigenvalues $\lambda$ of $M^{-1}A$. The converse is analogous. 
    
    $\emph{(1)} \Leftrightarrow \emph{(5)}$: This follows from \cref{prop:B-normal_spectral_radius}.  
\end{proof}

Using this result we also obtain the following Corollary. 

\begin{corollary}
    Assume that the smoothing assumption is satisfied. Then the matrices $\tildeMinvB$ and $\hatMinvB$ as in \cref{eq:tildeMinvB} and \cref{eq:hatMinvB}, respectively, are nonsingular for any choice of nonsingular $M \in \Cnn$ such that $M^{-1}A$ is $B$-normal.
\end{corollary}

With \cref{prop:B-normal_spectral_radius} we obtain the following.
\begin{corollary}\label{cor:spectral_radius smoother}
    Let $M^{-1}A$ be $B$-normal. Then $\norm{I-M^{-1}A}_B = \rho(I-M^{-1}A)$.
\end{corollary}

\begin{remark}
    For the case $A=B$ HPD and $M$ Hermitian, this result is already known \cite[Lemma 2.11]{BenFroNabSzy2001}. Then it holds 
    \begin{equation}\label{eq:spectral_radius HPD case}
        \| I - M^{-1}A \|_A = \rho(I - M^{-1}A). 
    \end{equation}
    Given this assumptions we have $(M^{-1}A)^+ = M^{-H}A = M^{-1}A$ and $M^{-1}A$ is $B$-orthogonal, in particular also $B$-normal and we obtain \cref{eq:spectral_radius HPD case} from \cref{cor:spectral_radius smoother}.
\end{remark}

Recall that we consider the eigenvalue problems
\begin{displaymath}
    M^{-1}Az = \lambda z \quad \text{and} \quad \hatMinvB v = \mu v.
\end{displaymath}
If the smoothing assumption is satisfied, the matrix $\widehat M^{-1}$ is nonsingular and both eigenvalue problems can be equivalently rewritten as the generalized eigenvalues problems
\begin{align}
    Az &= \lambda Mz, \label{eq:gen eig val prob MinvA}
    \\ Bv &= \mu \widehat M v. \label{eq:gen eig val prob hatMinvB}
\end{align}
In \cite{AliBraKahKrzSchSou2025, KrzSouWimAliBraKah2025} the first generalized eigenvalue problem is considered. \cite{BatNab2025} studies the second (generalized) eigenvalue problem with $\widetilde M^{-1}$ rather than $\widehat M^{-1}$. However, we restrict ourselves to $\hatMinvB$ here, because the eigenvalue problems are the same if $M^{-1}A$ is $B$-normal. 

Moreover, if $M^{-1}A$ is $B$-normal, the proof of \cref{cor:equiv_smoothing_assump} suggests that there exists a relation between the eigenvalues of $M^{-1}A$ and $\hatMinvB$. This connection is formalized in the following result. 

\begin{theorem}\label{thm:charact_eigenvals_tildeMB}
    Let $M^{-1}A$ be $B$-normal and $(\lambda,z)$ be an eigenpair of $M^{-1}A$. Then $(\mu,z)$ with 
    \begin{displaymath}
        \mu = \lambda + \overline{\lambda} - \overline{\lambda}\lambda = 1- \abs{\lambda-1}^2
    \end{displaymath}
    is an eigenpair of $\hatMinvB = \tildeMinvB$. Furthermore, $M^{-1}A$ and $\hatMinvB = \tildeMinvB$ are simultaneously $B$-unitarily diagonalizable.
\end{theorem}
\begin{proof}
    By \cref{thm:B-normal_charact}~(\labelcref{thm:enum:B-unitarily diag}) the matrix $M^{-1}A$ is $B$-unitarily diagonalizable, that is $M^{-1}A = W\Lambda W^{-1}$ for a $B$-unitary matrix $W \in \Cnn$ and diagonal matrix $\Lambda \in \Cnn$ containing the eigenvalues. By \cref{thm:B-normal_charact}~(\labelcref{thm:enum:eigenvector A Aplus}) the matrix $(M^{-1}A)^+$ is then also diagonalizable with eigenvectors $W$ and eigenvalues $\Lambda^H$, i.e. $(M^{-1}A)^+ = W\Lambda^HW^{-1}$. Thus we have
    \begin{align*}
        \hatMinvB &= M^{-1}A + (M^{-1}A)^+ -  M^{-1}A(M^{-1}A)^+ = W(\Lambda + \Lambda^H - \Lambda \Lambda^H)W^{-1}.
    \end{align*}
    By \cref{prop:properties tildeMinvB hatMinvB} it holds $\hatMinvB = \tildeMinvB$. Hence, $M^{-1}A$ and $\hatMinvB = \tildeMinvB$ are simultaneously $B$-unitarily diagonalizable and each eigenpair $(\lambda, z)$ of $M^{-1}A$ is an eigenpair $(\mu, z)$ of $\hatMinvB$ with $\mu = \lambda + \lambdaquer - \lambda \lambdaquer = 1- \abs{\lambda -1}^2$.
\end{proof}

Under the above assumptions, for $\sigma(M^{-1}A) = \{\lambda_1, \dots, \lambda_n\}$ we obtain 
\begin{displaymath}
    \sigma(\hatMinvB) = \{1 - \abs{\lambda_1 - 1}^2, \dots, 1- \abs{\lambda_n-1}^2\}.
\end{displaymath}
The function $f(z) = 1-\abs{z-1}^2$ that describes the eigenvalues of $\hatMinvB$ is maximal if and only if $\abs{z-1}^2= 0$ which is fulfilled for $z=1$. The zeros of $f$ are given by all $z \in \C$ with $\abs{z-1} = 1$.

In the HPD case Lemma \ref{thm:charact_eigenvals_tildeMB} gives the relation between the eigenpairs of $M^{-1}A$ (smoother times system matrix) and 
$\widetilde{M}^{-1}A$ (symmetrized smoother times system matrix) if $M^{-1}A$ is $A$-normal. This condition is fulfilled if $M$ is Hermitian. Thus, if $A$ is HPD and $M$ is Hermitian we do have the relation between the eigenvalue problems, in particular the eigenvectors are the same. But in general these eigenvalue problems are different.

The next auxiliary result connects the eigenvectors of $\hatMinvB$ to the $B$-orthogonality of $\Pi_A$.

\begin{lemma}\label{lem:range P kernel R B-orthogonal}
    Let $V= \m{v_1 & \cdots & v_n}$ be the $B$-unitary eigenvectors of $\hatMinvB$. Consider the following statements:
    \begin{enumerate}[(1)]
        \item $\calR(P) = \calR(\m{v_1 & \cdots & v_{n_c}})$.
        \item $\calN(R^HA) = \calR(\m{v_{n_c+1} & \cdots & v_n})$.
        \item $\Pi_A(P,R)$ is $B$-orthogonal.
    \end{enumerate}
    Any two of the three conditions imply the third statement. 
\end{lemma}
\begin{proof}
    $\emph{(1)} \land \emph{(2)} \Rightarrow \emph{(3)}$: Assume that $\emph{(1)}$ and $\emph{(2)}$ hold, that is $P = V  \big[ \begin{smallmatrix} I_{n_c} \\ 0 \end{smallmatrix} \big] Z$ for some nonsingular $Z \in \Cncnc$ and $R^HAV  \big[ \begin{smallmatrix} 0 \\ I_{n_f} \end{smallmatrix} \big] S = 0$ for all nonsingular $S \in \Cncnc$. It holds
    \begin{align}\label{eq:Pi_A B-unitary diag calc}
        \Pi_A(P,R) &= P(R^HAP)^{-1}R^HA = V  \m{I_{n_c} \\ 0} Z \left( R^HAV  \m{I_{n_c} \\ 0} Z\right)^{-1} R^HAVV^{-1} \nonumber
        \\ &= V  \m{I_{n_c} \\ 0} \left(R^HAV \m{I_{n_c} \\ 0} \right)^{-1}\m{R^HAV \m{I_{n_c} \\ 0} & 0} V^{-1}
        \\ &= V \m{I_{n_c} \\ 0}  \m{I_{n_c} & 0}V^{-1} = V  \m{I_{n_c} & 0 \\ 0 & 0 } V^{-1}. \nonumber
    \end{align}
    Since $V$ is $B$-unitary, it holds $V^HBV = I_n$ and thus $B = V^{-H}V^{-1}$. Hence, we get
    \begin{align*}
        \Pi_A(P,R)^+ = B^{-1}\Pi_A(P,R)^HB = VV^HV^{-H}\m{I_{n_c} & 0 \\ 0 & 0 }V^HV^{-H}V^{-1} = \Pi_A(P,R)
    \end{align*}
    and $\Pi_A(P,R)$ is $B$-orthogonal. 

    $\emph{(1)} \land \emph{(3)} \Rightarrow \emph{(2)}$: Assume that $\emph{(1)}$ and $\emph{(3)}$ are true. By \cref{lem:equiv_cond Pi B-ortho} it holds 
    \begin{displaymath}
        \calN(R^HA) = \calN(P^HB) = \calN(\m{I_{n_c} & 0}V^HB).
    \end{displaymath}
    Writing any $x \in \Cn$ as $x = \sum_{i=1}^n \alpha_i v_i$ with coefficients $\alpha_i \in \C$, we obtain
    \begin{displaymath}
        \m{I_{n_c} & 0}V^HBx = \m{I_{n_c} & 0} \sum_{i=1}^n \alpha_i V^HBv_i = \m{I_{n_c} & 0} \sum_{i=1}^n \alpha_i e_i 
        = \m{\alpha_1 \\ \vdots \\ \alpha_{n_c} \\ 0}.
    \end{displaymath}
    Thus $x \in \calN(\m{I_{n_c} & 0}V^HB)$ if and only if $\alpha_1 = \dots = \alpha_{n_c} = 0$ and hence 
    \begin{displaymath}
        \calN(R^HA) = \calN(\m{I_{n_c} & 0}V^HB) = \calR(\m{v_{n_c+1} & \dots & v_n}).
    \end{displaymath}

    $\emph{(2)} \land \emph{(3)} \Rightarrow \emph{(1)}$: Assume that $\emph{(2)}$ and $\emph{(3)}$ are true. By \cref{lem:equiv_cond Pi B-ortho} it holds 
    \begin{displaymath}
        \calN(P^HB) = \calN(R^HA) = \calR(\m{v_{n_c+1} & \dots & v_n}).
    \end{displaymath}
    Since $V$ is nonsingular we can write $P = V S$ for some matrix $S = \big[ \begin{smallmatrix} S_{n_c} \\ \tilde S \end{smallmatrix} \big] \in \Cnnc$ with full rank $n_c$ where $S_{n_c} \in \Cncnc$ and $\tilde S \in \C^{n_f \times n_c}$. Then we get
    \begin{align*}
        0 = P^HB \m{v_{n_c+1} & \cdots & v_n} = S^HV^HBV \m{0 \\ I_{n_f}} = \m{S_{n_c}^H & \tilde S^H} \m{0 \\ I_{n_f}} = \tilde S^H.
    \end{align*}
    Therefore, $S_{n_c}$ must be nonsingular and 
    \begin{displaymath}
        \calR(P) = \calR(VS) = \calR(V\big[ \begin{smallmatrix} I_{n_c} \\ 0 \end{smallmatrix} \big] S_{n_c}) = \calR(\m{v_1 & \cdots & v_{n_c}}).
    \end{displaymath}
    
\end{proof}

\section{Optimal transfer operators}\label{sec:two-grid conv}
Recall the two error propagation matrices
\begin{align*}
    E_+^{\nu_1, \nu_2} &= (I-(M^{-1}A)^+)^{\nu_2}(I-P(R^HAP)^{-1}R^HA)(I-M^{-1}A)^{\nu_1}, \tag{\ref{eq:Eplus} revisited}
    \\ E^{\nu_1, \nu_2} &= (I-M^{-1}A)^{\nu_2}(I-P(R^HAP)^{-1}R^HA)(I-M^{-1}A)^{\nu_1}. \tag{\ref{eq:Esouth} revisited}
\end{align*}
The second error operator $\Esouth$ has been analyzed in \cite{KrzSouWimAliBraKah2025, Not2010}. For $\nu_1 = 1$ and $\nu_2 = 0$, we obtain the simplified operator $E_+^{1,0} = E^{1,0}$ that has been studied in \cite{BatNab2025, AliBraKahKrzSchSou2025, ManSou2019, Xu2022}. Therefore, all of the following results containing convergence bounds for any of the two error operators can be interpreted as generalizations of the result in \cite{BatNab2025, AliBraKahKrzSchSou2025, ManSou2019, Xu2022}.

In the HPD case, the convergence analysis and the construction of optimal transfer operators use the eigenvalues and eigenvectors of the symmetrized smoother matrix multiplied by the system matrix \cite{XuZik2017, BraCaoKahFalHu2018}. This idea was also used for nonsymmetric AMG in \cite{KrzSouWimAliBraKah2025, BatNab2025, AliBraKahKrzSchSou2025}. For our analysis of $\Eplus$ we follow a similar approach and use the eigenpairs of $\hatMinvB$.

\begin{theorem}\label{thm:opti_Bnorm_Eplus_BatNab}
    Let the smoothing assumption be satisfied. Let 
    \begin{displaymath}
        \hatMinvB = V \diag(\mu_1, \dots, \mu_n) V^{-1}
    \end{displaymath}
    be a $B$-unitary diagonalization of $\hatMinvB$ with sorted eigenvalues $\mu_1 \leq \dots \leq \mu_n$ and corresponding $B$-unitary eigenvectors $V = \m{v_1 & \cdots & v_n}$. Let $\hat P, \hat R \in \Cnnc$ be defined such that 
    \begin{displaymath}
        \calR(\hat P) = \calR(\m{v_1 & \cdots & v_{n_c}}) \quad \text{and} \quad \calN(\hat R^HA) = \calR(\m{v_{n_c+1} & \cdots & v_n}).
    \end{displaymath}
    For $(\nu_1, \nu_2) \in \{ (0,1),(1,0),(1,1)\}$ it then holds
    \begin{displaymath}
        \norm{\Eplus(\hat P,\hat R)}^2_B = \norm{\Eplus(P_*,\hat R)}^2_B = \norm{\Eplus(\hat P,R_*)}^2_B = (1- \mu_{n_c+1})^2 <1.
    \end{displaymath}
\end{theorem}

\begin{proof}
    Let the smoothing assumption be satisfied. By \cref{prop:properties tildeMinvB hatMinvB} the matrix $\hatMinvB$ is $B$-orthogonal and thus it is $B$-unitarily diagonalizable with real eigenvalues $\mu_1 \leq \dots \leq \mu_n$ and correspondingly sorted $B$-unitary eigenvector matrix $V$. The smoothing assumption together with \cref{thm:smoothing_assumpt_BatNab} imply that $0<\mu_1 \leq \dots \leq \mu_n \leq 1$. Due to \cref{lem:range P kernel R B-orthogonal} we know that the projection $\Pi_A(\hat P, \hat R)$ is $B$-orthogonal. Similarly, by the definition of $P_*$ and $R_*$, the projections $\Pi_A(P_*, \hat R)$ and $\Pi_A(\hat P, R_*)$ are also $B$-orthogonal. It holds $\calR(P_*) = \calR(\hat P)$ and by \cref{lem:range P kernel R B-orthogonal} we also have $\calN(R_*^HA) = \calN(\hat R^HA)$. Thus 
    \begin{displaymath}
        \Pi_A(\hat P, \hat R) = \Pi_A(\hat P, R_*) = \Pi_A(P_*, \hat R)
    \end{displaymath}
    and $\Eplus(\hat P,\hat R)$, $\Eplus(P_*,\hat R)$ and $\Eplus(\hat P,R_*)$ coincide. 
    
    It holds $\hat P = V \big[ \begin{smallmatrix} I_{n_c} \\ 0 \end{smallmatrix} \big] Z$ for some nonsingular $Z \in \Cncnc$. The definition of $\calN(\hat R^HA)$ implies that $\hat R^HA V \bigl[ \begin{smallmatrix} 0 \\ I_{n_f} \end{smallmatrix} \bigr] = 0$. The same calculation as in \cref{eq:Pi_A B-unitary diag calc} then shows
    \begin{displaymath}
        \Pi_A(\hat P, \hat R) = V \m{I_{n_c} & 0 \\ 0 & 0}V^{-1}.
    \end{displaymath}
    Now, $\hatMinvB$ and $\Pi_A(\hat P, \hat R)$ are simultaneously $B$-unitarily diagonalizable and therefore they commute (cf. \cite[Theorem~1.3.21]{HorJoh2013}). Finally, for any pair $(\hat P, \hat R)$, $(P_*, \hat R)$ or $(\hat P, R_*)$ we get
    \begin{align*}
        \norm{\Eplus}_B^2 &= \lambda_{max}\left((I-(M^{-1}A)^+)^{\nu_1}(I-\Pi_A)^+(I-M^{-1}A)^{\nu_2}(I-(M^{-1}A)^+)^{\nu_2} \right.
        \\ &\hspace*{2cm} \left. (I-\Pi_A)(I-M^{-1}A)^{\nu_1} \right)
        \\ &= \lambda_{max}((I-\hatMinvB)^{\nu_1}(I-\Pi_A)(I-\hatMinvB)^{\nu_2}(I-\Pi_A))
        \\ &= \lambda_{max}((I-\hatMinvB)^{\nu_1 +\nu_2}(I-\Pi_A))
        \\ &= \lambda_{max}\left((I-\diag(\mu_1, \dots, \mu_n))^{\nu_1+\nu_2}\left(I-\m{I_{n_c} &0\\0&0} \right)\right)
        \\ &= \lambda_{max} \left( \m{0_{n_c} & & & \\ & (1-\mu_{{n_c}+1})^{\nu_1+\nu_2} & & \\ & & \ddots & \\ & & & (1-\mu_{n})^{\nu_1+\nu_2}}\right) 
        \\&= (1-\mu_{{n_c}+1})^{\nu_1+\nu_2} < 1.
    \end{align*}
\end{proof}

This result directly generalizes \cite[Theorems~4.5,~4.6]{BatNab2025} where the same result is proven for $\nu_1 = 1$ and $\nu_2 = 0$. 

For $B=A$ HPD this result reduces to the known results \cite[Theorems~5.3, 5.6, 5.7]{XuZik2017}. \Cref{thm:opti_Bnorm_Eplus_BatNab} holds for any HPD matrix $B$. This might seem like a contradiction to \cref{thm:B-normal_charact}~(\labelcref{thm:enum:diag and decomp B}) which describes the set of all HPD matrices such that a given matrix can be $B$-normal. But is not the case. Because $\hatMinvB$ is $B$-orthogonal for any HPD matrix $B$ (see \cref{prop:properties tildeMinvB hatMinvB}), any change in $B$ results in a different matrices $\hatMinvB$, $\hat P$ and $\hat R$ and therefore also different error operators. Hence, if we use $M^{-1}A$ and $(M^{-1}A)^+$ as pre- and post-smoother, respectively, combined with $\Pi_A(\hat P, \hat R)$ as in \cref{thm:opti_Bnorm_Eplus_BatNab}, we are able to construct a convergent two-grid method in any given inner product on $\C^n$. 

\begin{remark}
    It is possible to achieve the same result for arbitrary $\nu_1, \nu_2 \in \N_0$ if we add the assumption that $M^{-1}A$ is $B$-normal.
\end{remark}

In the second part of our theoretical framework \cite{NabRoo2026b} we prove further characterizations for the $B$-norm of $\Eplus$.

In \cref{thm:opti_Bnorm_Eplus_BatNab} we considered the special choice $\hat P$ and $\hat R$ as prolongation and restriction operators. An obvious question is whether these choices of $P$ and $R$ are optimal in the sense that any other pair $(P,R)$ of interpolation and restriction yields a larger $B$-norm of $\Eplus(P,R)$ than $(\hat P, \hat R)$. To answer this question, we first require the generalized Courant-Fisher-Weyl theorem (cf. \cite[Theorem~3.2]{Li2015}).

\begin{theorem}\label{thm:gen C-F-W}
    Let $\lambda_1 \leq \dots \leq \lambda_n$ be the eigenvalues of the matrix $C^{-1}A$ where $A \in \Cnn$ is Hermitian and $C \in \Cnn$ is HPD. Then, for any subspace $M \subseteq \Cn$ it holds
    \begin{displaymath}
        \lambda_k = \min_{\substack{M \subseteq \Cn \\ \dim(M) = n-k+1} } \max_{\substack{x \in M \\ x \neq 0}} \frac{x^HAx}{x^HCx}.
    \end{displaymath}
\end{theorem}

With this result we can now show that $(\hat P, \hat R)$ is indeed the combination that yields the smallest $B$-norm of the error operator $\Eplus$ for $(\nu_1, \nu_2) \in \{ (0,1),(1,0),(1,1)\}$. The only additional constraint we require is the invertibility of $I-\hatMinvB$. The result also shows that the optimal transfer operators are not unique. The proof uses the same idea as the proof of \cite[Theorem~3.5]{KrzSouWimAliBraKah2025}.

\begin{theorem}\label{thm:optimalitytransferopEplus}
  Let the smoothing assumption be satisfied. Let 
    \begin{displaymath}
        \hatMinvB = V \diag(\mu_1, \dots, \mu_n) V^{-1} =: VDV^{-1}
    \end{displaymath}
    be a $B$-unitary diagonalization of $\hatMinvB$ with sorted eigenvalues $\mu_1 \leq \dots \leq \mu_n$ and corresponding $B$-unitary eigenvectors $V = \m{v_1 & \cdots & v_n}$. Furthermore, let $I-\hatMinvB$ be nonsingular and $\hat P, \hat R \in \Cnnc$ be defined so that 
    \begin{displaymath}
        \calR(\hat P) = \calR(\m{v_1 & \cdots & v_{n_c}}) \quad \text{and} \quad \calN(\hat R^HA) = \calR(\m{v_{n_c+1} & \cdots & v_n}).
    \end{displaymath}
    Then, for $(\nu_1, \nu_2) \in \{ (0,1),(1,0),(1,1)\}$ and all $P,R \in \Cnnc$ such that $R^HAP$ is nonsingular it holds
    \begin{displaymath}
        \norm{\Eplus(\hat P,\hat R)}^2_B = (1- \mu_{n_c+1})^{2} \leq \norm{\Eplus(P,R)}^2_B.
    \end{displaymath}
\end{theorem}

\begin{proof}
    Let the smoothing assumption be satisfied. By \cref{prop:properties tildeMinvB hatMinvB,thm:smoothing_assumpt_BatNab} the matrix $\hatMinvB$ is then $B$-unitarily diagonalizable with real eigenvalues $0 < \mu_1 \leq \dots \leq \mu_n \leq 1$ and correspondingly sorted $B$-unitary eigenvector matrix $V$. Furthermore, let $I-\hatMinvB$ be nonsingular. Since 
    \begin{displaymath}
        I- \hatMinvB = (I-M^{-1}A)(I-M^{-1}A)^+
    \end{displaymath}
    the matrix $I-M^{-1}A$ is also nonsingular. With $(C^+)^{-1} = B^{-1}C^{-H}B = (C^{-1})^+$ we get
    \begin{displaymath}
        (I-\hatMinvB)^{-1} = \left((I-M^{-1}A)(I-M^{-1}A)^+ \right)^{-1} = \left( (I-M^{-1}A)^{-1}\right)^+(I-M^{-1}A)^{-1}.
    \end{displaymath}
    The $B$-orthogonality of $V$ implies $B = V^{-H}V^{-1}$. For all $P,R \in \Cnnc$ such that $R^HAP$ is nonsingular it now holds
    \begin{align*}
        \norm{\Eplus(P,R)}^2_B &= \norm{(I-(M^{-1}A)^+)^{\nu_2}(I-\Pi_A(P,R))(I-M^{-1}A)^{\nu_1}}_B^2
        \\ &= \sup_{y \neq 0} \frac{\norm{(I-(M^{-1}A)^+)^{\nu_2}(I-\Pi_A(P,R))(I-M^{-1}A)^{\nu_1}y}^2_B}{\norm{y}_B^2}
        \\ &= \sup_{x \neq 0} \frac{\norm{(I-(M^{-1}A)^+)^{\nu_2}(I-\Pi_A(P,R))x}^2_B}{\norm{((I-M^{-1}A)^{-1})^{\nu_1}x}_B^2}
        \\ &\geq \sup_{\substack{x \in \calN(R^HA) \\ x \neq 0}} \frac{\norm{(I-(M^{-1}A)^+)^{\nu_2}x}^2_B}{\norm{((I-M^{-1}A)^{-1})^{\nu_1}x}_B^2}
        \\ &= \sup_{\substack{x \in \calN(R^HA) \\ x \neq 0}} \frac{\< (I-(M^{-1}A)^+)^{\nu_2}x, (I-(M^{-1}A)^+)^{\nu_2}x \>_B}{\<((I-M^{-1}A)^{-1})^{\nu_1}x, ((I-M^{-1}A)^{-1})^{\nu_1}x \>_B}
        \\ &= \sup_{\substack{x \in \calN(R^HA) \\ x \neq 0}} \frac{\< x, (I-\hatMinvB)^{\nu_2}x \>_B}{\< x ,((I-\hatMinvB)^{-1})^{\nu_1}x \>_B}
        \\ &= \sup_{\substack{x \in \calN(R^HA) \\ x \neq 0}} \frac{x^H(V^{-H}(I-D)V^{-1})^{\nu_2}x}{x^H(V^{-H}(I-D)^{-1}V^{-1})^{\nu_1}x} =: \sup_{\substack{x \in \calN(R^HA) \\ x \neq 0}} \frac{x^H\mathcal{A}^{\nu_2}x}{x^H \mathcal{C}^{\nu_1}x}
    \end{align*}
    where $\mathcal{C}$ is HPD and $\mathcal{A}$ is Hermitian. Then $\mathcal{C}^{-1}\mathcal{A} = V(I-D)^{\nu_1+\nu_2}V^{-1}$ and has eigenvalues $(1-\mu_i)^{\nu_1+\nu_2}$. By \cref{thm:gen C-F-W} it holds the following
    \begin{align*}
        \min_{P,R \in \Cnnc} \norm{\Eplus(P,R)}_B^2 &= \min_{P \in \Cnnc} \min_{R \in \Cnnc} \norm{\Eplus(P,R)}_B^2
        \\ &\geq \min_{P \in \Cnnc} \min_{R \in \Cnnc} \sup_{\substack{x \in \calN(R^HA) \\ x \neq 0}} \frac{x^H\mathcal{A}^{\nu_2}x}{x^H \mathcal{C}^{\nu_1}x}
        \\& = \min_{P \in \Cnnc} \min_{R \in \Cnnc} \sup_{\substack{x \in M \\ x \neq 0 \\ M = \calN(R^HA)}} \frac{x^H\mathcal{A}^{\nu_2}x}{x^H \mathcal{C}^{\nu_1}x}
        \\ &= \min_{P \in \Cnnc} \min_{\substack{M \subseteq \Cn \\ \dim(M) = n-n_c}} \sup_{\substack{x \in M \\ x \neq 0}} \frac{x^H\mathcal{A}^{\nu_2}x}{x^H \mathcal{C}^{\nu_1}x}
        \\&= \min_{P \in \Cnnc} (1-\mu_{n_c+1})^{\nu_1+\nu_2} = (1-\mu_{n_c+1})^{\nu_1+\nu_2}.
    \end{align*}
    Due to \cref{thm:opti_Bnorm_Eplus_BatNab} we know that this bound is attained for the pairs $(\hat P, \hat R)$, $(P_*, \hat R)$ and $(\hat P, R_*)$.
\end{proof}

For $(\nu_1, \nu_2) = (1,0)$ this result was already proven in \cite[Theorem~4.5,~4.6]{BatNab2025}.

\begin{remark}
    $(\hat P, \hat R)$, $(P_*, \hat R)$ and $(\hat P, R_*)$ are not the only combinations of interpolation and restriction that minimize the $B$-norm of $\Eplus$. In fact, any pair $(P,R)$ with $\calR(P) = \calR(\hat P)$ and $\calR(R) = \calR(\hat R)$ minimizes the $B$-norm of the operator $\Eplus$. This means that the optimal compatible transfer operators are unique up to a change of basis.
\end{remark}

The operator $\Eplus$ for $(\nu_1, \nu_2) \in \{ (0,1),(1,0),(1,1)\}$ provides a natural generalization of the HPD error operator $E_{\text{HPD}}^{\nu_1,\nu_2}$ and gives a convergent two-grid methods under fairly easy assumptions. However, from an application-oriented point of view, computing the post-smoother $(M^{-1}A)^+ = B^{-1}A^HM^{-H}B$ seems infeasible due to the inverse of $B$. But for special cases of $B$, the post-smoother can still be computed. For more details about the computation of $P_*$ in these cases, see \cite{BatNab2025}.

In the following  we list some special choices for $B$. For the classic HPD case, that is $B= A$ HPD, it holds
\begin{displaymath}
    (M^{-1}A)^+ = A^{-1}A^HM^{-H}A = M^{-H}A
\end{displaymath}
which yields the error propagation operator $E_{\text{HPD}}^{\nu_1, \nu_2}(P,P)$ when choosing $R = P$. In this case, \cref{thm:optimalitytransferopEplus} reduces to the known result \cite[Theorem~5.7]{XuZik2017}.

For $B= I_n$ we get 
\begin{displaymath}
    (M^{-1}A)^+ = (M^{-1}A)^H = A^HM^{-H}
\end{displaymath}
which can be computed easily. Moreover, we have
\begin{displaymath}
    \widehat M^{-1} = M^{-1}A + A^HM^{-H} -  M^{-1}AA^HM^{-H}.
\end{displaymath}

For $B = A^HA$ we get
\begin{displaymath}
    (M^{-1}A)^+ = A^{-1}A^{-H}A^HM^{-H}A^HA = A^{-1}M^{-H}A^HA = A^{-1}M^{-H}B
\end{displaymath}
but this smoother contains $A^{-1}$ which is also infeasible. However, if we choose additionally $M^{-1} = AA^H$ we obtain
\begin{displaymath}
    (M^{-1}A)^+ = A^{H}B
\end{displaymath}
The operator $R_*$ that produces a $B$-orthogonal coarse-grid correction can be easily determined as $R_* = A^{-H}BP = AP$ given some interpolation operator $P$. 

The case $B=M$ HPD was also investigated in \cite{Xu2022}. Here, we have 
\begin{displaymath}
    (M^{-1}A)^+ = M^{-1}A^HM^{-H}M = M^{-1}A^H
\end{displaymath}
which is a simple computation. \Cref{thm:opti_Bnorm_Eplus_BatNab} was proved in \cite[Theorem~3.5]{Xu2022} for the error operator $E^{1,0} = E_+^{1,0}$ and the special case for $B=M$ HPD.

For the last case, we consider a singular value decomposition (SVD) of $A$, i.e. $A = U\Sigma V^H$ with unitary matrices $U,V \in \Cnn$ and $\Sigma = \diag(\sigma_1, \dots, \sigma_n)$. Then define $Q = VU^H$ and $B = QA = V\Sigma V^H$ for which we have 
\begin{displaymath}
    (M^{-1}A)^+ = (QA)^{-1}A^HM^{-H}QA = QM^{-H}QA.
\end{displaymath}
Here, the restriction operator $R_*$ can be easily computed again, since $R_* = A^{-H}BP = UV^HP = Q^HP$ given some interpolation $P$. 

To avoid the computation of the post-smoother $(M^{-1}A)^+$, it seems natural to consider a simpler error operator, namely $\Esouth$. For $\nu_1 = 1$ and $\nu_2 = 0$, it holds $E_+^{1,0} = E^{1,0}$ and similar to before this case was already studied in \cite{AliBraKahKrzSchSou2025,BatNab2025}. For arbitrary $\nu_1$ and $\nu_2$, analysis has been done in \cite{KrzSouWimAliBraKah2025}. There, the general assumption was used that $M^{-1}A$ is diagonalizable. With this diagonalization, a certain class of HPD matrices $B$ is established and then used to show convergence results for the operator $\Esouth$ in this class of $B$-norms. The optimal transfer operators leading to the minimal $B$-norms are also given in \cite{KrzSouWimAliBraKah2025}.

We now show that the assumption of diagonalizability of $M^{-1}A$ and the resulting class of $B$-norms actually refer to $M^{-1}A$ being $B$-normal.

While \cref{thm:opti_Bnorm_Eplus_BatNab} utilizes the eigenvalues and eigenvectors of $\hatMinvB$, the idea used in \cite{AliBraKahKrzSchSou2025} and then later in \cite{KrzSouWimAliBraKah2025} was to build the optimal transfer operators with the eigenvectors of $M^{-1}A$. Note, that \cref{thm:charact_eigenvals_tildeMB} gives a relation between the eigenvectors and eigenvalues of $M^{-1}A$ and $\tildeMinvB$. The following result is proven in \cite[Lemma~4.1]{AliBraKahKrzSchSou2025}.

\begin{lemma}\label{lem:Ali_left_right_eigenvec}
    Let $A,M \in \Cnn$ be such that $M$ is invertible and $M^{-1}A$ is diagonalizable. Consider the left and right generalized eigenvectors $V_l,V_r \in \Cnn$, respectively, of the generalized eigenvalue problem \cref{eq:gen eig val prob MinvA} such that
    \begin{align}
        A V_r  &= M V_r \Lambda,  \label{eq:def right gen eig vec V_r}
        \\ V_l^H A &= \Lambda V_l^HM \label{eq:def left gen eig vec V_l}
    \end{align}
    where $\Lambda \in \Cnn$ is a diagonal matrix containing the eigenvalues. Then $V_l$ and $V_r$ induce a matrix-based orthogonality, satisfying
    \begin{align}
        V_l^H A V_r  &= D_A,  \label{eq:left_right_eig A ortho} 
        \\ V_l^H M V_r  &= D_M  \label{eq:left_right_eig M ortho}
    \end{align}
    for nonsingular diagonal matrices $D_A, D_M \in \Cnn$. 
\end{lemma}

Note that this Lemma can also be proven by using the Kronecker canonical form of matrix pairs (cf. \cite{LanZab2009,LanRod2005}). 

Since $M$ is nonsingular, $V_r$ are the right generalized eigenvectors of the generalized eigenvalues problem \cref{eq:gen eig val prob MinvA} if and only if $V_r$ are eigenvectors of $M^{-1}A$. It also holds
\begin{displaymath}
    \Lambda = V_r^{-1} M^{-1}A V_r = (V_l^H M V_r)^{-1} V_l^H A V_r = D_M^{-1}D_A.
\end{displaymath}
Based on the right eigenvectors $V_r$ of $M^{-1}A$, \cite[Theorem~3.1]{KrzSouWimAliBraKah2025} characterizes all possible HPD matrices $B$ such that the coarse grid projection can be $B$-orthogonal. This structure of $B$ is also discussed in \cite[Section~3]{KrzSouWimAliBraKah2025}. The following result demonstrates that this structure actually refers to $B$-normality and its stronger version of $B$-orthogonality. 

\begin{lemma}\label{lem:relation B-normal KrzSouWimAliBraKah2025}
Let $M^{-1}$ be diagonalizable with $M^{-1}A = W\Lambda W^{-1}$ and let $\Pi_A$ be a projection with $\dim(\calR(\Pi_A)) = n_c$. Then $\Pi_A$ is diagonalizable with $\Pi_A = S \big[ \begin{smallmatrix} I_{n_c} & 0 \\ 0 & 0 \end{smallmatrix}\big] S^{-1}$ and the following statements holds.
\begin{enumerate}[1)]
    \item $M^{-1}A$ is $B$-normal if and only if $B = (WDW^H)^{-1}$ for an HPD block diagonal matrix $D$ where the block sizes correspond to those in $\Lambda$. 
    \item $\Pi_A$ is $B$-orthogonal if and only if $B = (SDS^H)^{-1}$ for an HPD block diagonal matrix $D = \big[ \begin{smallmatrix} D_{n_c} & 0 \\ 0 & D_{n_f} \end{smallmatrix}\big]$ with two HPD blocks $D_{n_c} \in \Cncnc$ and $D_{n_f} \in \C^{n_f \times n_f}$. \label{lem:enum:equiv cond Pi_A B-ortho}
\end{enumerate}

\end{lemma}
\begin{proof}
$\emph{1)}$ follows from \cref{thm:B-normal_charact}~(\labelcref{thm:enum:diag and decomp B}) under the general assumption that $M^{-1}A$ is diagonalizable. For $\emph{2)}$ we observe that $\Pi_A$ is always diagonalisable with $\Pi_A = S \big[ \begin{smallmatrix} I_{n_c} & 0 \\ 0 & 0 \end{smallmatrix}\big] S^{-1}$ where the identity block is the size of the dimension of $\calR(\Pi_A)$, that is $n_c$, \cite[Theorem 2.11] {Gal2004}. Therefore, $\Pi_A$ has real eigenvalues and the result follows from $\emph{1)}$ and \cref{prop:B-orthogonal_iff_eig_real}.
\end{proof}

When we apply this result to the specific diagonalization $M^{-1}A = V_r \Lambda V_r^{-1}$ with $V_r$ as in \cref{eq:def right gen eig vec V_r}, we can conclude that the assumption that $M^{-1}A$ is diagonalizable used in \cite{AliBraKahKrzSchSou2025, KrzSouWimAliBraKah2025} leads to $B$-normal smoothers for specific HPD matrices $B$.

With this, we now present \cite[Theorem~3.3]{KrzSouWimAliBraKah2025} in the framework of $B$-normal matrices. Notice that we also extended the class of possible inner products and condensed the proof of the result. 

\begin{theorem}\label{thm:opti_Bnorm_South}
    Let $M^{-1}A$ be $B$-normal and 
    \begin{displaymath}
        V_r = \m{v_1^{(r)} & \cdots & v_n^{(r)}} \quad \text{and} \quad V_l = \m{v_1^{(l)} & \cdots & v_n^{(l)}}
    \end{displaymath}
    be the right and left eigenvectors of the generalized eigenvalue problem \cref{eq:gen eig val prob MinvA}, respectively, such that \cref{eq:def right gen eig vec V_r} and \cref{eq:def left gen eig vec V_l} are satisfied with the eigenvalue matrix $\Lambda = \big[ \begin{smallmatrix} \Lambda_{n_c} & 0 \\ 0 & \Lambda_{n_f} \end{smallmatrix}\big]$. Assume that the eigenvalues $\lambda$ are ordered such that 
    \begin{equation}\label{eq:ordereigSouth}
    \abs{1 - \lambda_1} \geq  \abs{1 - \lambda_2} \geq \dots \geq \abs{1 -\lambda_n}
    \end{equation}
    and equal eigenvalues for a single diagonal block in $\Lambda$. Let $B = V_r^{-H}D^{-1}V_r^{-1}$ where
    \begin{equation}\label{eq:def D opti_B_norm_South}
        D = \m{D_{n_c} & 0 \\0 &D_{n_f}} \in \Cnn
    \end{equation}
    is HPD block diagonal with $D_{n_c} \in \C^{n_c\times n_c}$ and $D_{n_f} \in \C^{n_f\times n_f}$. Both $D_{n_c}$ and $D_{n_f}$ are also HPD block diagonal with block sizes corresponding to those in $\Lambda_{n_c}$ and $\Lambda_{n_f}$. Furthermore, let $P_\#, R_\# \in \Cnnc$ be defined such that
    \begin{align*}
        \calR(P_\#) = \calR\left(\m{v^{(r)}_1 & \cdots & v^{(r)}_{n_c}}\right) \quad \text{and} \quad \calR(R_\#) = \calR\left(\m{v^{(l)}_1 & \cdots & v^{(l)}_{n_c}}\right).
    \end{align*} 
    Then $\Pi_A(P_\#, R_\#)$ is $B$-orthogonal and it holds
    \begin{displaymath}
         \rho(\Esouth(P_\#,R_\#)) = \norm{\Esouth(P_\#,R_\#)}_B = \abs{1 - \lambda_{n_c + 1}}^{\nu_1 + \nu_2}.
    \end{displaymath}
\end{theorem}
\begin{proof}
    Let $M^{-1}A$ be $B$-normal with right and left generalized eigenvectors $V_r$ and $V_l$ such that \cref{eq:def right gen eig vec V_r} and \cref{eq:def left gen eig vec V_l} are satisfied, respectively. Let the eigenvalues $\lambda$ be ordered as in \cref{eq:ordereigSouth}, let $B= V_r^{-H}D^{-1}V_r^{-1}$ with $D$ as in \cref{eq:def D opti_B_norm_South} and $P_\#$ and $R_\#$ defined as above. We partition the diagonal matrix $D_A \in \Cnn$ from \cref{eq:left_right_eig A ortho} similar to $\Lambda$, that is
    \begin{equation}\label{eq:def D_A opti_Bnorm_Esouth}
        D_A = \m{D_{n_c}^{(A)} & 0 \\ 0 & D_{n_f}^{(A)}}
    \end{equation}
    with $D_{n_c}^{(A)} \in \Cncnc$ and $D_{n_f}^{(A)} \in \C^{n_f \times n_f}$. By the definitions of $P_\#$ and $R_\#$ we know that
    $P_\# = V_r \big[ \begin{smallmatrix} I_{n_c} \\ 0 \end{smallmatrix}\big] Z$ and $R_\# = V_l \big[ \begin{smallmatrix} I_{n_c} \\ 0 \end{smallmatrix}\big] S$ hold for some nonsingular matrices $Z, S \in \Cncnc$. It follows
    \begin{align}\label{eq:Pi PhashRhash diagonalizable}
        \Pi_A(P_\#,R_\#) &= V_r \m{I_{n_c} \\ 0} Z \left( S^H \m{I_{n_c} & 0} V_l^H A V_r \m{I_{n_c} \\ 0} Z\right)^{-1} S^H \m{I_{n_c} & 0} V_l^HA \nonumber
        \\ &= V_r \m{I_{n_c} \\ 0} \left( \m{I_{n_c} & 0} D_A \m{I_{n_c} \\ 0} \right)^{-1} \m{I_{n_c} & 0} D_AV_r^{-1} \nonumber 
        \\ &= V_r \m{I_{n_c} \\ 0} (D_{n_c}^{(A)})^{-1} \m{I_{n_c} & 0} D_A V_r^{-1} 
        \\ &= V_r \m{(D_{n_c}^{(A)})^{-1} & 0 \\ 0 & 0} D_A V_r^{-1}= V_r \m{I_{n_c} & 0 \\ 0 &0}V_r^{-1}. \nonumber
    \end{align}
    By \cref{lem:relation B-normal KrzSouWimAliBraKah2025}~(\labelcref{lem:enum:equiv cond Pi_A B-ortho}) and the structure of $B$ the projection $\Pi_A(P_\#,R_\#)$ is $B$-orthogonal. Since $M^{-1}A = V_r\Lambda V_r^{-1}$, the matrices $M^{-1}A$ and $\Pi_A(P_\#,R_\#)$ are simultaneously diagonalizable and therefore commute. Thus, we have
    \begin{align*}
        \Esouth(P_\#,R_\#) &= (I-M^{-1}A)^{\nu_2}(I-\Pi_A(P_\#, R_\#))(I-M^{-1}A)^{\nu_1}
        \\ &= (I-\Pi_A(P_\#, R_\#))(I-M^{-1}A)^{\nu_1+\nu_2}
        \\ &= V_r \left(I - \m{I_{n_c} & 0 \\ 0 & 0} \right)V_r^{-1}V_r\left( I-\Lambda\right)^{\nu_1+\nu_2} V_r^{-1} 
        \\ &= V_r \m{0 & 0 \\ 0 & I_{n_f} - \Lambda_{n_f}}^{\nu_1+\nu_2}V_r^{-1}
        \\ &= V_r D^{\frac12} \m{0 & 0 \\ 0 & I_{n_f} - \Lambda_{n_f}}^{\nu_1+\nu_2}D^{-\frac12}V_r^{-1}.
    \end{align*}
    Since $D$ is HPD, $D^{\frac12}$ exists and the matrix $V_r D^\frac12$ is $B$-unitary. Then the error operator $\Esouth(P_\#, R_\#)$ is $B$-unitarily diagonalizable and hence $B$-normal. With \cref{prop:B-normal_spectral_radius} and the ordering of the eigenvalues~\cref{eq:ordereigSouth} we obtain
    \begin{displaymath}
         \norm{\Esouth(P_\#,R_\#)}_B = \rho(\Esouth(P_\#,R_\#)) = \abs{1-\lambda_{n_c + 1}}^{\nu_1 + \nu_2}.
    \end{displaymath}
\end{proof}

\begin{remark}
    The result itself is independent of the smoothing assumption. However, to obtain a convergent two-grid method with $\norm{\Esouth(P_\#,R_\#)}_B < 1$ the smoothing assumption is required.
\end{remark}

In \cite{NabRoo2026b} we prove further characterizations of the $B$-norm of the error operator $\Esouth$.

The proof of \cref{thm:opti_Bnorm_South} shows that $V_rD^{\frac12}$ is $B$-unitary. Since $D$ is nonsingular, it holds $\calR(V_r) = \calR(V_rD^{\frac12})$ and $V_rD^\frac12$ is a $B$-unitary eigenvector matrix of $M^{-1}A$. With \cref{thm:charact_eigenvals_tildeMB} we infer that $\hatMinvB$ is also $B$-unitarily diagonalizable with eigenvector matrix $V_rD^\frac12$ and eigenvalues $1-\abs{1-\lambda}^2$. The block structure of $B$ shows that
\begin{displaymath}
    \calR(P_\#) = \calR(\m{v_1 &\cdots&v_{n_c} }) = \calR(\hat P)
\end{displaymath}
and with \cref{lem:range P kernel R B-orthogonal} we obtain
\begin{displaymath}
    \calN(R_\#^HA) = \calR(\m{v_{n_c+1} &\cdots & v_{n} }) = \calN(\hat R^HA).
\end{displaymath}
The $B$-orthogonality of $\Pi_A(\hat P, \hat R)$ and \cref{lem:equiv_cond Pi B-ortho} then yield
\begin{align*}
    \calR(\hat R) &= \calR(A^{-H}B\hat P) = \calR(A^{-H}BP_\#) = \calR \left(A^{-H}V_r^{-H}D^{-1}V_r^{-1}V_r \m{I_{n_c}\\0} \right)
    \\ &= \calR \left(A^{-H}V_r^{-H} \m{I_{n_c}\\0} \right) = \calR \left( V_l D_A \m{I_{n_c}\\0} \right) = \calR \left( V_l\m{I_{n_c}\\0} \right) = \calR(R_\#).
\end{align*}
Additionally, the sorting of the eigenvalues $\lambda$ of $M^{-1}A$ as in \cref{eq:ordereigSouth} implies the ordering
\begin{displaymath}
    \mu_1 = 1- \abs{1-\lambda_1}^2 \leq \mu_2 = 1- \abs{1-\lambda_2}^2 \leq \cdots \leq \mu_n = 1-\abs{1-\lambda_n}^2 
\end{displaymath}
of the eigenvalues $\mu$ of $\hatMinvB$. Therefore, for $(\nu_1,\nu_2) \in \{(0,1),(1,0),(1,1)\}$ and given the assumption of \cref{thm:opti_Bnorm_South}, all the assumptions of \cref{thm:opti_Bnorm_Eplus_BatNab} are satisfied and we obtain
\begin{displaymath}
    \norm{\Eplus(P_\#,R_\#)}_B^2 = (1-\mu_{n_c+1})^{\nu_1+\nu_2} = \abs{1-\lambda_{n_c+1}}^{2(\nu_1+\nu_2)} = \norm{\Esouth(P_\#,R_\#)}_B^2.
\end{displaymath}
In summary, for $(\nu_1,\nu_2) \in \{(0,1),(1,0),(1,1)\}$ and $M^{-1}A$ is $B$-normal, \cref{thm:opti_Bnorm_Eplus_BatNab,thm:opti_Bnorm_South} gives the same (up to a change of basis) restriction and interpolation operators and the same $B$-norm estimate. The theorems differ in the assumption on $M^{-1}A$ and in the number of smoothing steps. While \cref{thm:opti_Bnorm_Eplus_BatNab} is true for any $B$-norm and does not need assumptions on $M^{-1}A$, it requires $(\nu_1,\nu_2) \in \{(0,1),(1,0),(1,1)\}$. On the other hand, \cref{thm:opti_Bnorm_South} is valid for arbitrary numbers of pre- and post-smoothing steps but requires the $B$-normality of $M^{-1}A$ which restricts the set of all possible $B$-norms. 

The structure of the matrix $D$ and therefore also $B$ might seem complicated, but it can be explained using \cref{thm:B-normal_charact}~(\labelcref{thm:enum:diag and decomp B}). The interpolation and restriction $P_\#$ and $R_\#$ are defined independently of $B$ and by \cref{thm:opti_Bnorm_South} we know that $\Pi_A(P_\#, R_\#)$ is $B$-orthogonal for any matrix $B = V_r^{-H} D_1^{-1} V_r^{-1}$ where $D_1$ is an HPD block diagonal matrix with two HPD blocks, one of size $n_c$ and one of size $n_f$ (cf. \cref{lem:relation B-normal KrzSouWimAliBraKah2025}~(\labelcref{lem:enum:equiv cond Pi_A B-ortho})). This explains the HPD block $D_{n_c}$ in \cref{eq:def D opti_B_norm_South}. For the other block $D_{n_f}$ we apply \cref{thm:B-normal_charact}~(\labelcref{thm:enum:diag and decomp B}) to $M^{-1}A$ and obtain the structure $B = V_r^{-H}D_2^{-1} V_r^{-1}$ where $D_2$ is now an HPD block diagonal matrix with block sizes corresponding to those of $\Lambda$. Hence, the $B$-normality of $M^{-1}A$ explains the structure of the lower right block $D_{n_f}$ in \cref{eq:def D opti_B_norm_South}. In short, for $M^{-1}A$ to be $B$-normal and $\Pi_A$ to be $B$-orthogonal, the $B$-inner product must be of the form $B = W^{-H}DW^{-1}$ where $W$ are eigenvectors of $M^{-1}A$ \emph{and} $\Pi_A$ and $D$ is an HPD block diagonal matrix with block sizes corresponding to the block sizes of the eigenvalues of $M^{-1}A$ \emph{and} $\Pi_A$.

Similar to \cref{thm:optimalitytransferopEplus}, the question remains whether $(P_\#, R_\#)$ is the optimal combination of interpolation and restriction operator in the sense that they yield the smallest possible $B$-norm of $\Esouth$. This question was already answered in \cite[Theorem~3.5]{KrzSouWimAliBraKah2025} for diagonal HPD matrices $D$ (instead of those in \cref{eq:def D opti_B_norm_South}). Here we state their result and extend the class of $B$-inner products by allowing block HPD matrices $D$ as in \cref{eq:def D opti_B_norm_South}. The proof for diagonal $D$ can be found in \cite{KrzSouWimAliBraKah2025} and we give no proof here because it is analogous to the proof in \cite{KrzSouWimAliBraKah2025} and the proof of \cref{thm:optimalitytransferopEplus}. 

\begin{theorem}\label{thm:optimalitytransferopEsouth}
    Let the assumptions from \cref{thm:opti_Bnorm_South} be fulfilled, the smoothing assumption be satisfied and $I-M^{-1}A$ be nonsingular. Then it holds
    \begin{displaymath}
        \norm{\Esouth(P_\#,R_\#)}_B  = \abs{1- \lambda_{n_c+1}}^{\nu_1 + \nu_2} \leq \norm{\Esouth(P,R)}_B 
    \end{displaymath}
    for all $P, R \in \Cnnc$ such that $R^HAP$ is nonsingular.
\end{theorem}
For $(\nu_1,\nu_2) = (1,0)$ the optimality of the transfer operators already follows from \cite[Theorem~4.6]{BatNab2025} by using the structure of the matrix $B$ defined by the HPD block diagonal matrix $D$ as in \cref{eq:def D opti_B_norm_South} and the $B$-normality of $M^{-1}A$ together with \cref{thm:charact_eigenvals_tildeMB}.

\subsection{Some remarks}
To emphasize the differences of the two error operators $\Eplus$ and $\Esouth$ and their respective assumptions for convergence, we consider some theoretical examples. These examples shall improve the understanding of the required assumptions and the different structures of the error operators. 

In the \cref{thm:opti_Bnorm_South,thm:optimalitytransferopEsouth} we considered certain $B$-norms with matrices $B$ that are constructed with block diagonal HPD matrices $D$ of the form \cref{eq:def D opti_B_norm_South}. In the underlying theory from \cite[Theorems~3.3,~3.5]{KrzSouWimAliBraKah2025} only diagonal HPD matrices $D$ were allowed. The following example shows a case where \cite[Theorem~3.3]{KrzSouWimAliBraKah2025} does not guarantee convergence, but \cref{thm:opti_Bnorm_South} does. 

\begin{example}
    We consider $n =3$ and $n_c = 1$. Let
    \begin{displaymath}
        A = \m{\frac78 & -\frac58 & \frac58 \\[1mm] -\frac58 & \frac78 & \frac58 \\[1mm] 0 &0 & \frac32} = \m{1&0&1 \\ 1&1&0 \\ 0&1&1} \m{\frac14 & & \\ & \frac32 & \\ & & \frac32}\m{1&0&1 \\ 1&1&0 \\ 0&1&1}^{-1}, M^{-1} = I_3,
    \end{displaymath}
    then $M^{-1}A = A$ is diagonalizable with eigenvalues $\lambda_1 = \frac14$ and $ \lambda_{2} = \lambda_3 = \frac32$. The (generalized) right eigenvalues $V_r$ are given by 
    \begin{displaymath}
        V_r = \m{1&0&1 \\ 1&1&0 \\ 0&1&1}.
    \end{displaymath}
    The (generalized) left eigenvectors are the (generalized) right eigenvectors of $A^H = A^T$ which are given by 
    \begin{displaymath}
        V_l = \m{-1&-1&0 \\ -1&1&0 \\ 1&0&1}.
    \end{displaymath}
    Let $B$ be given by 
    \begin{displaymath}
        B = \m{\frac58&-\frac38&-\frac18 \\[1mm] -\frac38&\frac58&-\frac18 \\[1mm] -\frac18&-\frac18&\frac{7}{24}}.
    \end{displaymath}
    Then $B$ can be written as $B = V_r^{-H} D^{-1} V_r^{-1} = (V_rDV_r^H)^{-1}$ for 
    \begin{displaymath}
        D = \m{2&0&0 \\ 0&2&1 \\ 0&1&2}
    \end{displaymath}
    which is an HPD block diagonal matrix with block sizes corresponding to those of the eigenvalue matrix of $A$, that is a $2 \times 2$ block corresponding to the double eigenvalue $\frac32$ and a $1 \times 1$ block corresponding to the eigenvalue $\frac14$. By construction, the matrix $M^{-1}A$ is $B$-normal (a calculation shows that it even is $B$-orthogonal). We define
    \begin{displaymath}
        P_\# = \m{1\\1\\0} \quad \text{ and } \quad R_\# = \m{-1\\-1\\1}
    \end{displaymath}
    and obtain
    \begin{displaymath}
        \Pi_A(P_\#, R_\#) = \m{\frac12 &\frac12&-\frac12 \\[1mm] \frac12 & \frac12 & -\frac12 \\[1mm] 0 & 0 &0}
    \end{displaymath}
    which is $B$-orthogonal since $\Pi_A(P_\#, R_\#)^+ = \Pi_A(P_\#, R_\#)$. The eigenvalues $\lambda$ are ordered so that $\abs{1-\lambda_1} \geq \abs{1-\lambda_2} \geq \abs{1-\lambda_3}$. Then we can apply \cref{thm:opti_Bnorm_South} and obtain 
    \begin{displaymath}
        \rho(E^{\nu_1, \nu_2}(P_\#, R_\#) = \norm{E^{\nu_1, \nu_2}(P_\#, R_\#)}_B = \abs{1-\lambda_2}^{\nu_1 +\nu_2} = \frac{1}{2^{\nu_1 +\nu_2}}.
    \end{displaymath}
    
    This result cannot be obtained by \cite[Theorem~3.3]{KrzSouWimAliBraKah2025} because $D$ is not a diagonal matrix but has a block diagonal form. 
\end{example}

The next example illustrates a case where $M^{-1}A$ is not $B$-normal for any HPD matrix $B$ and therefore our convergence analysis only guarantees the convergence of $\Eplus$ with $(\nu_1,\nu_2) \in \{(0,1),(1,0),(1,1)\}$ but not of $\Esouth$. The example uses a matrix $M^{-1}A$ which is not diagonalizable and thus cannot be $B$-normal. 

\begin{example}
    We consider $n = 3$ and $n_c = 1$ again. Let $(\nu_1, \nu_2) = (1,1)$ and
    \begin{displaymath}
        A = \m{\frac12 &0 & 0\\ 0& \frac12 & 1 \\ 0& 0& \frac12}, \quad M^{-1} = I_3 \quad \text{and} \quad B = \m{1 & 0 & 0 \\ 0&1&-2\\0&-2&6}
    \end{displaymath}
    then $M^{-1}A = A$ is not diagonalizable and $B$ is HPD. A calculation shows
    \begin{displaymath}
        \widehat M^{-1} = M^{-1}AB^{-1} + B^{-1}A^HM^{-H}- M^{-1}AB^{-1}A^HM^{-H} = \m{\frac34 &0&0 \\ 0 &\frac{11}{4}&1 \\ 0 &1& \frac{3}{8}} 
    \end{displaymath}
    is HPD and thus the smoothing assumption is satisfied, that is $\norm{I-M^{-1}A}_B < 1$. Furthermore, it holds
    \begin{align*}
        \widehat M^{-1}B &= \m{\frac34 & 0 &0 \\[1mm] 0&\frac34 & \frac12 \\[1mm] 0&\frac14&\frac14}
        \\ &= \m{1&0&0\\0&-\sqrt{3}+1&\sqrt{3}+1\\0&1&1} \m{\frac34 & & \\ & \frac{-\sqrt{3}+2}{4} & \\ & & \frac{\sqrt{3}+2}{4}} \m{1&0&0\\0&-\sqrt{3}+1&\sqrt{3}+1\\0&1&1}^{-1}
    \end{align*}
    thus $\widehat M^{-1}B$ has the sorted eigenvalues $\mu_1 = \frac{-\sqrt{3}+2}{4} \leq \mu_2 = \frac34 \leq \mu_3 = \frac{\sqrt{3}+2}{4}$. We choose the interpolation $P = \Big[ \begin{smallmatrix} 0 \\ -\sqrt{3}+1\\1 \end{smallmatrix}\Big]$ as eigenvector of $\hatMinvB$ corresponding to $\mu_1$ and the corresponding optimal restriction $R_* = A^{-H}BP = \Big[ \begin{smallmatrix} 0 \\ -2\sqrt{3}-2\\8\sqrt3 +12 \end{smallmatrix}\Big]$. Then by construction of $P$ and $R_*$ we know that $\Pi_A(P,R_*)$ is $B$-orthogonal. A computation also shows
    \begin{displaymath}
        \Pi_A(P,R_*) = \m{0 & 0 & 0\\[1mm] 0&\frac{-\sqrt3 +3}{6} & \frac{-\sqrt3}{3} \\[1mm] 0 & \frac{-\sqrt3}{6} & \frac{\sqrt{3}+3}{6}} 
        \quad \text{ and } \quad
        E_+^{1,1} = \m{\frac14 & 0 & 0 \\[1mm] 0& \frac{\sqrt3 -1}{8} & \frac{-2\sqrt3 +3}{4} \\[1mm] 0& \frac{2\sqrt3 -3}{24} & \frac{-3\sqrt3 +5}{8}}
    \end{displaymath}
    with $\norm{E_+^{1,1}}_B = 1-\mu_2 = \frac14$ by \cref{thm:opti_Bnorm_Eplus_BatNab}. However, since $M^{-1}A$ is not $B$-normal, theory does not guarantee the convergence of $E^{1,1}$. However, it also holds
    \begin{displaymath}
        E^{1,1} = \m{\frac14 & 0 & 0 \\[1mm] 0 & \frac{-\sqrt3 +3}{24} & \frac{\sqrt3 -2}{4} \\[1mm] 0 & \frac{\sqrt3}{24} & \frac{-\sqrt3 +1}{8}} \quad \text{with} \quad \norm{E^{1,1}}_B = \frac{1}{4}.
    \end{displaymath}
    It is interesting that both of the operators have the same norm.
\end{example}

The example demonstrated that the $B$-normality of $M^{-1}A$ is not necessary for the convergence of $\Esouth$, however the following example shows that the operator requires at least some additional assumption for convergence. 

\begin{example}
    We consider $n=3$ and $n_c = 1$ and $(\nu_1, \nu_2) = (1,1)$. Let
    \begin{displaymath}
        A = \m{\frac14 & 0 & \frac1{12} \\[1mm] 0 & \frac12 & 0 \\[1mm] 0& 0& \frac14} = \m{1 & 0 &1\\0&1&0\\0&0&1} \m{\frac14 & 0&0\\ 0&\frac12&0\\0&0&\frac13}\m{1 & 0 &1\\0&1&0\\0&0&1}^{-1} = V_r\Lambda V_r^{-1}
    \end{displaymath}
    and $M=I_3$. Further let
    \begin{displaymath}
        B = \m{4 & 0&0\\0&2&1\\0&1&1}.
    \end{displaymath}
    Then $B$ is HPD and $M^{-1}A = A$ is diagonalizable. However, $M^{-1}A$ is not $B$-normal because $B$ does not have the required structure as in \cref{thm:B-normal_charact}~(\labelcref{thm:enum:diag and decomp B}). We have
    \begin{displaymath}
        \hatMinvB = \m{\frac{55}{144} &\frac{1}{36} & \frac{5}{72} \\[1mm] -\frac16 & \frac56 & \frac1{12} \\[1mm] \frac49 & -\frac29 & \frac49}
    \end{displaymath}
    which has eigenvalues in the interval $(0,1]$. By \cref{thm:smoothing_assumpt_BatNab} the smoothing assumption is then satisfied. Choosing $\hat P$ and $\hat R$ as in \cref{thm:opti_Bnorm_Eplus_BatNab}, the theorem also implies the convergence of $E_+^{1,1}$, that is $\norm{E_+^{1,1}(\hat P, \hat R)}_B < 1$. But the error operator $E^{1,1}$ does not convergence. The left eigenvectors of $M^{-1}A$ are given by
    \begin{displaymath}
        V_l = \m{1 &0&0\\0&1&0\\-1&0&1}.
    \end{displaymath}
    We set $P_\# = \Big[\begin{smallmatrix} 1 \\0\\0 \end{smallmatrix} \Big]$ and $R_\# = \Big[\begin{smallmatrix} 1 \\0\\-1 \end{smallmatrix} \Big] $ and obtain
    \begin{displaymath}
        E^{1,1}(P_\#, R_\#) = \m{\frac34 & 0 &-\frac1{12} \\[1mm] 0&\frac12&0\\[1mm] 0&0&\frac23} \m{0&0&1\\0&1&0\\0&0&1} \m{\frac34 & 0 &-\frac1{12} \\[1mm] 0&\frac12&0\\[1mm] 0&0&\frac23} = \m{0&0&\frac49 \\[1mm] 0&\frac14 &0\\[1mm] 0&0&\frac49}
    \end{displaymath}
    with 
    \begin{displaymath}
        \norm{E^{1,1}(P_\#, R_\#)}_B = \sqrt{\frac{\sqrt{1294465}+1217}{1296}} \approx 1.34794 > 1.
    \end{displaymath}
    Therefore, $E^{1,1}$ does not convergence. This is due to the fact that $\norm{\Pi_A(P_\#, R_\#)}_B = 3 > 1$ and $\Pi_A(P_\#, R_\#)$ is not $B$-orthogonal for this choice of interpolation and restriction. Increasing the $(1,1)$-entry of $B$ further increases the $B$-norm of $\Pi_A(P_\#, R_\#)$.
\end{example}

The last two examples show that additional assumptions are necessary to obtain the convergence of $E^{1,1}$. The $B$-normality of $M^{-1}A$ is not a necessary assumption, but it is a sufficient condition for the convergence. 

\section{Conclusions}
In this paper, we present a unified framework for the convergence analysis of algebraic two-grid methods in arbitrary norms given by HPD matrices $B$ for nonsymmetric and indefinite matrices. Our framework includes the approaches from \cite{BatNab2025, KrzSouWimAliBraKah2025, AliBraKahKrzSchSou2025} and generalizes the approach in \cite{BraCaoKahFalHu2018} for HPD matrices. We introduce the error propagation operator $\Eplus$ that acts as the natural generalization of the error operator in the HPD case. We also present the error operator $\Esouth$ which has a simpler structure than $\Eplus$ and was already studied in \cite{KrzSouWimAliBraKah2025}. For both operators, we prove the convergence for certain norms, for $\Eplus$ in arbitrary norms and for $\Esouth$ for a special class of norms defined by a diagonalization of the smoothing matrix $M^{-1} $ multiplied with the system matrix $A$. We state the optimal compatible transfer operators leading to the minimal error norms and discuss the similarities of these transfer operators for the two error operators. Our framework is continued in \cite{NabRoo2026b} where we further characterize the convergence rate of the two error operators $\Eplus$ and $\Esouth$. Moreover, we show that an increase of the number of coarse variables leads to a decrease of the error norm and thus faster convergence. Finally, in \cite{NabRoo2026b} we extend our framework to V-cycle multigrid methods and show that McCormick's V-cycle bound also holds true for nonsymmetric AMG methods.

\bibliographystyle{siamplain}
\bibliography{literature_paper}

\end{document}